\documentclass[preprint,11pt]{elsarticle}

\usepackage{amsfonts, amsmath, amscd}
\usepackage[psamsfonts]{amssymb}

\usepackage{amssymb}

\usepackage{pb-diagram}

\usepackage[all,cmtip]{xy}

\usepackage[usenames]{color}

\newtheorem{theorem}{Theorem}[section]
\newtheorem{lemma}[theorem]{Lemma}
\newtheorem{corollary}[theorem]{Corollary}

\newtheorem{proposition}[theorem]{Proposition}
\newtheorem{definition}[theorem]{Definition}
\newtheorem{example}[theorem]{Example}

\newtheorem{problem}[theorem]{Problem}

\newproof{proof}{Proof}

\numberwithin{equation}{section}

\newcommand{\e}{\varepsilon}
\newcommand{\w}{\omega}

\newcommand{\NN}{\mathbb{N}}

\newcommand{\IR}{\mathbb{R}}

\newcommand{\II}{\mathbb{I}}

\newcommand{\yyy}{\mathbf{y}}
\newcommand{\zzz}{\mathbf{z}}

\newcommand{\FF}{\mathcal{F}}

\newcommand{\V}{\mathcal{V}}
\newcommand{\U}{\mathcal{U}}
\newcommand{\W}{\mathcal{W}}

\newcommand{\UU}{\mathcal{U}}

\newcommand{\KK}{\mathcal{K}}
\newcommand{\Nn}{\mathcal{N}}
\newcommand{\AAA}{\mathcal A}
\newcommand{\I}{\mathcal{I}}

\newcommand{\CC}{C_k}

\newcommand{\SM}{{\setminus}}

\input xy
\xyoption{all}

\begin{document}

\begin{frontmatter}

\title{Ascoli and sequentially Ascoli spaces}

\author{S.~Gabriyelyan}
\ead{saak@math.bgu.ac.il}
\address{Department of Mathematics, Ben-Gurion University of the Negev, Beer-Sheva, P.O. 653, Israel}

\begin{abstract}
A Tychonoff space $X$ is called ({\em sequentially}) {\em Ascoli} if every compact subset (resp. convergent sequence) of $\CC(X)$ is evenly continuous, where $\CC(X)$ denotes the space of all real-valued continuous functions on $X$ endowed with the compact-open topology.  Various properties of (sequentially) Ascoli spaces are studied, and we give several characterizations of sequentially Ascoli spaces. Strengthening a result of Arhangel'skii we show that a hereditary Ascoli space is Fr\'{e}chet--Urysohn. 
A locally compact abelian group $G$ with the Bohr topology is sequentially Ascoli iff $G$ is compact. If $X$ is totally countably compact or near sequentially compact then it is a sequentially Ascoli space. 
The product of a locally compact space and an Ascoli space is Ascoli. If additionally $X$ is a $\mu$-space, then $X$ is locally compact iff the product of $X$ with any Ascoli space is an Ascoli space. Extending one of the main results of \cite{GGKZ} and  \cite{Gabr-B1} we show that $C_p(X)$ is sequentially Ascoli iff $X$ has the property $(\kappa)$. We give a necessary condition on $X$ for which the space $\CC(X)$ is sequentially Ascoli.
For every metrizable abelian group $Y$, $Y$-Tychonoff space $X$, and nonzero countable ordinal $\alpha$, the space $B_\alpha(X,Y)$  of Baire-$\alpha$ functions from $X$ to $Y$ is $\kappa$-Fr\'{e}chet--Urysohn  and hence Ascoli.
\end{abstract}

\begin{keyword}
$C_p(X)$ \sep $\CC(X)$ \sep Baire-$\alpha$ function \sep  Ascoli \sep  sequentially Ascoli \sep  $\kappa$-Fr\'{e}chet--Urysohn \sep $P$-space

\MSC[2010]  46E40 \sep 54A05 \sep  54B05 \sep   54C35 \sep 54D20

\end{keyword}

\end{frontmatter}



\section{Introduction}


Various topological properties generalizing metrizability have been  intensively studied  both by topologists and analysts  for  a long time. Especially important  properties are those ones which generalize metrizability: the Fr\'{e}chet--Urysohn property, sequentiality, the $k$-space property and countable tightness.  These properties are intensively studied for function spaces, free locally convex spaces, $(LM)$-spaces, strict $(LF)$-spaces and their strong duals, and Banach and Fr\'{e}chet spaces in the weak topology etc., see for example \cite{Arhangel,Gabr,Jar,kaksax,mcoy,S-W} and references therein.

A much weaker notion than Fr\'{e}chet--Urysohness was introduced by Arhangel'skii and studied in \cite{LiL}: a topological space $X$ is said to be {\em $\kappa$-Fr\'{e}chet--Urysohn} if for every open subset $U$ of $X$ and every $x\in \overline{U}$, there exists a sequence $\{ x_n\}_{n\in\w} \subseteq U$ converging to $x$.

For Tychonoff  topological spaces $X$ and $Y$, we denote by $\CC(X,Y)$ and $C_p(X,Y)$ the space $C(X,Y)$ of all continuous functions from $X$ to $Y$ endowed with the compact-open topology or the pointwise topology, respectively. If $Y=\IR$, we set $\CC(X):=\CC(X,\IR)$ and $C_p(X):=C_p(X,\IR)$. Being motivated by the classical Ascoli theorem we introduced in \cite{BG} a new class of topological spaces, namely, the class of Ascoli spaces. A Tychonoff space $X$  is {\em Ascoli} if every compact subset $\KK$ of $\CC(X)$  is evenly continuous, that is the map $X\times\KK \ni(x,f)\mapsto f(x)\in\IR$ is continuous. By Ascoli's theorem \cite[Theorem 3.4.20]{Eng}, each $k$-space is Ascoli.
The Ascoli property in topological spaces, topological groups, (free) locally convex spaces, function spaces etc. was intensively studied in \cite{Banakh-Survey,BG,Gabr-C2,Gabr-LCS-Ascoli,Gab-LF,Gabr-B1,Gabr-ind-lim,GGKZ,GKP}.
Being motivated by the classical notion of $c_0$-barrelled locally convex spaces, in \cite{Gabr:weak-bar-L(X)} we defined a Tychonoff space $X$ to be {\em sequentially Ascoli} if every convergent sequence in $\CC(X)$ is equicontinuous. Clearly, every Ascoli space is sequentially Ascoli, but the converse is not true in general, see \cite{Gabr:weak-bar-L(X)}. Below we formulate some of the most interesting results (although the clauses (i)--(iv) were proved for the property of being an Ascoli space, their proofs and Proposition \ref{p:Ascoli-sufficient} show that one can replace ``Ascoli'' by ``sequentially Ascoli'').

\begin{theorem} \label{t:Ascoli-seq-Ascoli-Banach}
\begin{enumerate}
\item[{\rm(i)}] {\rm (\cite{Gabr-LCS-Ascoli})} A Fr\'{e}chet space $E$ with the weak topology is a sequentially Ascoli space if and only if $E=\mathbb{F}^n$ or $E=\mathbb{F}^{\w}$, where $\mathbb{F}=\IR$ or $\mathbb{C}$ is the field of $E$.
\item[{\rm(ii)}] {\rm (\cite{Banakh-Survey,GKP})} The closed unit ball of a Banach space $E$ endowed with the weak topology is a sequentially Ascoli space if and only if $E$ does not contain an isomorphic copy of $\ell_1$.
\item[{\rm(iii)}] {\rm (\cite{Gab-LF})} A strict $(LF)$-space $E$ is a sequentially Ascoli space if and only if $E$ is a Fr\'{e}chet space or $E=\phi$.
\item[{\rm(iv)}] {\rm (\cite{Gab-LF})} The space of distributions $\mathcal{D}'(\Omega)$ is not a sequentially Ascoli space.
\item[{\rm(v)}] {\rm (\cite{Gabr-B1,GGKZ})} $C_p(X)$ is an Ascoli space if and only if $X$ has the property $(\kappa)$.
\end{enumerate}
\end{theorem}

The following diagram describes the relationships between the aforementioned properties:
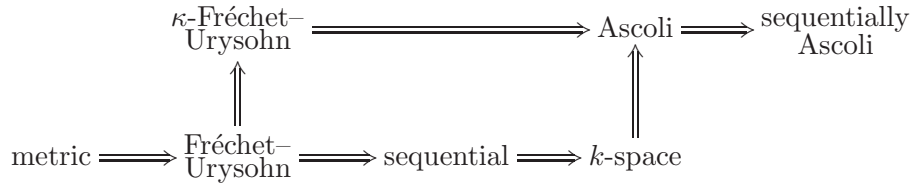
\[
\xymatrix{
& {\substack{\mbox{$\kappa$-Fr\'{e}chet--} \\ \mbox{Urysohn}}} \ar@{=>}[rr] & & {\mbox{Ascoli}} \ar@{=>}[r] & {\substack{\mbox{sequentially} \\ \mbox{Ascoli}}}\\
{\mbox{metric}} \ar@{=>}[r] & {\substack{\mbox{Fr\'{e}chet--} \\ \mbox{Urysohn}}} \ar@{=>}[r] \ar@{=>}[u] & {\mbox{sequential}} \ar@{=>}[r]  &  {\mbox{$k$-space}} \ar@{=>}[u] &
}
\]
None of these implications is reversible, see \cite{BG,Eng,Gabr-B1,Gabr:weak-bar-L(X)}. For further generalizations of Ascoli spaces see \cite{Banakh-g-Ascoli}.

Now we describe the content of the paper whose main purpose is the further study of Ascoli and sequentially Ascoli spaces. In Section \ref{sec:seq-Y-Ascoli} we establish some categorical properties of sequentially Ascoli spaces analogous to the corresponding properties of Ascoli spaces considered in \cite{BG}. In particular, various characterizations of sequentially Ascoli spaces are given in Theorem \ref{t:seq-R-Ascoli}. We show that the square of a  Fr\'{e}chet--Urysohn space can be not Ascoli (Proposition \ref{p:square-non-Ascoli}). Strengthening a result of Arhangel'skii we prove in Theorem \ref{t:her-Ascoli-k-space} that a hereditary Ascoli space is Fr\'{e}chet--Urysohn, however every non-discrete $P$-space is hereditary sequentially Ascoli but  Ascoli (Proposition \ref{p:P-space-2-Ascoli}).

In Section \ref{sec:compact-Ascoli} we show that a locally compact abelian group $G$ endowed with the Bohr topology is sequentially Ascoli if and only if $G$ is compact (Theorem \ref{t:seq-Ascoli-G+}). In Theorem \ref{t:seq-compact-seq-Ascoli} we prove that if $X$ is totally countably compact or near sequentially compact, then it is a sequentially Ascoli space. However  there are countably compact spaces which are not sequentially Ascoli (Proposition \ref{p:countably-compa-not-Ascoli}).

In Section \ref{sec:product-Ascoli} we prove that the product of a locally compact space and an Ascoli space is Ascoli (Theorem \ref{t:product-lc-Ascoli}). In Theorem \ref{t:Ascoli-product-mu} we partially reverse this result by showing that  a $\mu$-space $X$ is locally compact if and only if the product of $X$ with an arbitrary Ascoli (even Fr\'{e}chet--Urysohn) space is an Ascoli space.

In the last Section \ref{sec:seq-Ascoli-Cp} we study various function spaces which are sequentially Ascoli. In Theorem \ref{t:Cb-Ascoli} we extend (v) of Theorem \ref{t:Ascoli-seq-Ascoli-Banach} by showing that $C_p(X)$ is sequentially Ascoli if and only if $X$ has the property $(\kappa)$. Let $Y$ be a metrizable abelian group, $X$ be a $Y$-Tychonoff space, and let $\alpha$ be a nonzero countable ordinal.  In Theorem \ref{t:Baire-kFU} we show that the space $B_\alpha(X,Y)$  of all Baire-$\alpha$ functions from $X$ to $Y$ is $\kappa$-Fr\'{e}chet--Urysohn  and hence Ascoli.
To obtain these results we actively use the idea successfully applied in  \cite{Gabr-B1} and which is that instead of the whole function space $C_p(X,Y)$ we consider only some of its sufficiently rich and saturated subspaces in the sense of Definition \ref{def:relatively-Y-Tych}.
Finally in Proposition \ref{p:seq-Ascoli-Ck-necessary} we obtain a necessary condition on $X$ for which the space $\CC(X)$ is sequentially Ascoli.


\section{Sequentially $Y$-Ascoli spaces} \label{sec:seq-Y-Ascoli}


In this section we characterize sequentially Ascoli spaces and prove some standard categorical assertions. Also we provide several examples which clarify relationships between the considered notions and pose open questions.

Recall that a family $\I$ of compact subsets of a topological space $X$ is called an {\em ideal of compact sets} if $\bigcup\I=X$ and for any sets $A,B\in\I$ and any compact subset $K\subseteq X$ we get $A\cup B\in\I$ and $A\cap K\in\I$. Let $\FF(X)$ and $\KK(X)$ be the family of all finite subsets or all compact subsets of $X$, respectively. Denote by $\mathcal{S}(X)$ the family of all finite unions of convergent sequences in $X$. Clearly, $\FF(X)$, $\mathcal{S}(X)$ and $\KK(X)$ are  ideals of compact sets in $X$ and $\FF(X)\subseteq \mathcal{S}(X)\subseteq \KK(X)$.

We shall consider the following separation axioms.
\begin{definition} \label{def:relatively-Y-Tych}{\em
Let $Y$ be a topological space. A topological space $X$ is called
\begin{enumerate}
\item[$\bullet$] (\cite{BG-Baire}) {\em $Y$-Tychonoff} if for every closed subset $A$ of $X$, point $y_0\in Y$ and function $f:F\to Y$ defined on a finite subset  $F$ of $X\setminus A$, there exists a continuous function ${\bar f}:X\to Y$ such that ${\bar f}{\restriction}_F=f$ and ${\bar f}(A)\subseteq \{ y_0\}$;
\item[$\bullet$] (\cite{BG-Baire}) {\em $Y$-normal} if $X$ is a $T_1$-space and for any closed set $F\subseteq X$ and each continuous function $f:F\to Y$ with finite image $f(F)$ there exists a continuous function ${\bar f}:X\to Y$ such that ${\bar f}|_F=f$;
\item[$\bullet$] (\cite{GO}) {\em $Y_\I$-Tychonoff} for an ideal $\I$ of compact subsets of $X$ if for every closed subset $A$ of $X$, point $y_0\in Y$ and function $f:K\to Y$ with finite image $f(K)$ defined on a subset  $K\in\I$ of $X\setminus A$, there exists a continuous function ${\bar f}:X\to Y$ such that ${\bar f}{\restriction}_K=f$ and ${\bar f}(A)\subseteq \{ y_0\}$.
\end{enumerate} }
\end{definition}
For $\I=\FF(X)$, $\mathcal{S}(X)$ or $\KK(X)$ we shall say that $X$ is $Y_p$-Tychonoff, $Y_s$-Tychonoff or $Y_k$-Tychonoff, respectively.
Clearly, $X$ is $Y_p$-Tychonoff if and only if it is $Y$-Tychonoff, and
\[
\xymatrix{
\mbox{$Y_k$-Tychonoff}  \ar@{=>}[r] & \mbox{$Y_\I$-Tychonoff}\ar@{=>}[r] & \mbox{$Y$-Tychonoff}.  }
\]
If $Y$ is path-connected and admits a non-constant continuous function $\chi:Y\to\IR$, then, by Proposition 2.9 of \cite{GO}, $X$ is Tychonoff if and only if it is $Y_k$-Tychonoff. 

Let $X$ and $Y$ be topological spaces. For every $y\in Y$, we denote by $\yyy\in Y^X$ the constant function defined  by $\yyy(x):=y$ for every $x\in X$.  Each ideal $\I$ of compact subsets of $X$ determines the {\em $\I$-open topology $\tau_\I$} on the power space $Y^X$. A subbase of this topology consists of the sets
\[
[K;U]=\{f\in Y^X :f(K)\subseteq U\},
\]
where $K\in\I$ and $U$ is an open subset of $Y$. In particular, for every $g\in Y^X$, finite subfamily $\FF=\{F_1,\dots,F_n\} \subseteq \I$ and each family $\U=\{ U_1,\dots,U_n\}$ of open subsets of $Y$ such that $g(F_i)\subseteq U_i$ for every $i=1,\dots,n$, the sets of the form
\[
W[g;\FF,\U] :=\bigcap_{i=1}^n W[g;  F_i,U_i], \; \mbox{ where } \; W[g;  F_i,U_i]:= \big\{ f\in Y^X: f(F_i)\subseteq U_i\big\} .
\]
form  a base of the $\I$-open topology $\tau_\I$ at the function $g$. The space  $C(X,Y)$ of all continuous functions from $X$ to $Y$ endowed with the $\I$-open topology induced from $(Y^X,\tau_\I)$ will be denoted by $C_\I(X,Y)$. If $\I=\FF(X), \mathcal{S}(X)$ or $\KK(X)$, we write $\tau_\I =\tau_p, \tau_s$ or $\tau_k$ and $C_\I(X,Y)=C_p(X,Y), C_s(X,Y)$ or $\CC(X,Y)$, respectively.
In the case when $Y=\IR$ we write simply $C_p(X)$, $C_s(X)$ or $\CC(X)$ and observe that the sets of the form
\[
[K;\e]:=[K;(-\e,\e)],
\]
where $K\subseteq X$ belongs to $\I$ and $\e>0$, form a base at the zero function $\mathbf{0}$ of the topology $\tau_\I$. In the case when $Y$ is a discrete abelian group and also for simplicity of notations, we set
\[
[K;0] :=[K;\{0\}].
\]

\begin{proposition} \label{p:Ts=Tp=Tk}
Let $Y$ be a $T_1$-space containing at least two points.  Then:
\begin{enumerate}
\item[{\rm (i)}] $\tau_p\leq\tau_s\leq \tau_k$.
\item[{\rm (ii)}] If $X$ is a $Y$-Tychonoff space, then $\tau_s=\tau_p$ if and only if every convergent sequence in $X$ is essentially constant.
\item[{\rm (iii)}] If $X$ is a $Y_s$-Tychonoff space, then $\tau_s=\tau_k$ if and only if  each compact subset of $X$ is contained in a finite union of convergent sequences.
\end{enumerate}
\end{proposition}

\begin{proof}
(i) is clear. Below we prove only (ii) because (iii) has a similar proof.

(ii) Assume that $\tau_s=\tau_p$ and suppose for a contradiction that $X$ contains an infinite convergent sequence $S=\{ x_n\}_{n\in\w}$ with $x_n\to x_0$.
Since $|Y|>1$, fix two distinct points $z,t\in Y$. As $Y$ is $T_1$, choose an open neighborhood $V$ of $z$ such that $t\not\in V$. Then $[S;V]$ is a $\tau_s$-neighborhood of the constant function $\zzz$.
To get a contradiction (i.e. that $\tau_p\not=\tau_s$), we show that $W[\zzz;F,W]\nsubseteq W[\zzz;S,V]$
for every finite subset  $F$ of $X$ and each open neighborhood $W\subseteq Y$ of $z$. We assume that $W\subseteq V$. Indeed, since $S$ is infinite, there exists an $m\in\w$ such that $x_m\not\in F$. As $X$ is $Y$-Tychonoff, there is $f\in C(X,Y)$ such that $f(F)\subseteq W$ and $f(x_m)\not\in V$. Then $f\in W[\zzz;F,W]\setminus W[\zzz;S,V]$.
Conversely, if every convergent sequence in $X$ is essentially constant then the equality $\tau_s=\tau_p$ holds trivially.\qed
%
\end{proof}

For a cardinal number $\kappa$, let $V(\kappa)$ be the \emph{Fr\'{e}chet-Urysohn fan} which is, by definition, the topological space obtained from $\kappa$ many copies of pairwise disjoint one-to-one convergent sequences by identifying their limit points, endowed with the quotient topology.
Specifically, $V(\kappa) = \bigcup_{i \in \kappa}S_i$, where each $S_i$ is homeomorphic to a convergent sequence with its limit $x_\infty$ and $S_i \cap S_j = \{ x_\infty \}$ whenever $i \ne j$.
All points of $V(\kappa)$ except for $x_\infty$ are isolated, while a neighborhood of $x_\infty$ is any $U \subseteq V(\kappa)$ such that $x_\infty \in U$ and $S_i \setminus U$ is finite for every $i\in\kappa$. It is well-known (and easy to check) that $V(\kappa)$ is a Fr\'{e}chet--Urysohn Tychonoff space.
Note that every compact subset of $V(\kappa)$ has empty intersection with all but finitely many sets of the form $S_i^- = S_i\setminus \{x_\infty\}$.

\begin{example} {\em
Let $Y=\IR$. Now, if $X=\beta\NN$, then $\tau_p=\tau_s<\tau_k$. If $X=V(\kappa)$, 
then $\tau_p<\tau_s =\tau_k$. For $X=[0,1]$, we have $\tau_p<\tau_s<\tau_k$.\qed}
\end{example}

Recall that a subset $A$ of $\CC(X,Y)$ is called {\em evenly continuous} if the evaluation map $\psi: A\times X\to Y, \psi(f,x):=f(x)$, is continuous, i.e., for every $f\in A$, $x\in X$ and neighborhood $U_{f(x)}$ of $f(x)$ there are neighborhoods $V_f \subseteq A$ of $f$ in $A\subseteq \CC(X,Y)$ and $O_x$ of $x$ such that $\psi(g,z)=g(z)\in U_{f(x)}$ for every $g\in V_{f}$ and each $z\in O_x$.

\begin{definition} {\em
Let $X$ and $Y$ be Tychonoff spaces. The space $X$ is called a {\em sequentially $Y$-Ascoli space} if each convergent sequence $S\subseteq \CC(X,Y)$ is evenly continuous. }
\end{definition}
In Theorem \ref{t:seq-R-Ascoli} below we show that  $X$ is a sequentially Ascoli space if and only if $X$ is sequentially $\IR$-Ascoli.
Also we shall repeatedly use without additional mentioning the following remark: Since a convergent sequence $S\subseteq \CC(X,Y)$ has only one non-isolated point, say $f_0$, to check that $S$ is evenly continuous it is sufficient to prove that $S$ is evenly continuous at any point $(f_0,x)$ with $x\in X$.

To characterize sequentially $Y$-Ascoli spaces we need additional notations.
For any topological spaces $X$ and $Y$ the canonical valuation map
\[
\delta:X\to C(\CC(X,Y),Y), \quad \delta(x)(f):=f(x),
\]
assigns to each point $x\in X$ the {\em $Y$-valued Dirac measure} $\delta(x):\CC(X,Y)\to Y$ concentrated at $x$.
The map $\delta$ is well-defined since the function $\delta(x):\CC(X,Y)\to Y$ is continuous at each function $f\in \CC(X,Y)$ (indeed, for any open neighborhood $V\subseteq Y$ of $\delta(x)(f)=f(x)$, the set $[\{x\};V]$ is an open neighborhood of $f$ in $\CC(X,Y)$ with $\delta(x)\big([\{x\};V]\big)\subseteq V$).

All topological groups are assumed to be $T_0$ and hence Tychonoff.
Let $Y$ be an abelian  topological group, and let $X$ be a topological space. A subset $\FF$ of $C(X,Y)$ is called {\em equicontinuous} if for every $x\in X$ and each neighborhood $U$ of zero $0\in Y$ there is a neighborhood $O_x$ of $x$ such that
\[
f(t)-f(x)\in U \; \mbox{ for all }\; t\in O_x \; \mbox{ and }\; f\in\FF.
\]
It is clear that every equicontinuous set is evenly continuous, but the converse is not true in general.



We proved in \cite{BG} that $X$ is $Y$-Ascoli if and only if the canonical map $\delta$ is continuous as a map from $X$ to $\CC(\CC(X,Y),Y)$. Below we obtain an analogous characterization of sequentially $Y$-Ascoli spaces.

\begin{theorem}\label{t:seq-Ascoli}
Let $Y$ be a Tychonoff space containing at least two points,  $X$ be a $Y$-Tychonoff space, and let $\delta: X\to C_s\big(\CC(X,Y),Y\big)$ be the canonical map. Then the following assertions are equivalent:
\begin{enumerate}
\item[{\rm (i)}] $X$ is a sequentially $Y$-Ascoli space;
\item[{\rm (ii)}] $\delta$ is continuous;
\item[{\rm (iii)}]  $\delta$ is an embedding;
\item[{\rm (iv)}] each point $x\in X$ is contained in a dense sequentially $Y$-Ascoli subspace of $X$.
\end{enumerate}
If in addition $Y$ is an abelain group, then {\rm (i)-(iv)} are equivalent to
\begin{enumerate}
\item[{\rm (v)}] every convergent sequence in $\CC(X,Y)$ is equicontinuous.
\end{enumerate}
\end{theorem}

\begin{proof}
(i)$\Rightarrow$(ii) Assume that $X$ is sequentially $Y$-Ascoli. We have to show that $\delta$ is continuous at each point $x_0\in X$. Fix a sub-basic neighborhood
\[
[S;V]\subseteq C_s(\CC(X,Y),Y)
\]
of $\delta(x_0)$, where $S=\{ f_n:n\in\w\}\subseteq \CC(X,Y)$ with $f_n\to f_0$  and $V\subseteq Y$ open. It follows from $\delta(x_0)\in[S;V]$ that $f_n(x_0)=\delta(x_0)(f_n)\in V$ for every $n\in \w$. Since $X$ is sequentially $Y$-Ascoli,
there are $N\in\w$ and an open neighborhood $\mathcal{O}_0$ of $x_0$ such that $f_n(\mathcal{O}_0)\subseteq V$ for every $n>N$. Taking a smaller neighborhood $\mathcal{O}\subseteq X$ of $x_0$ we can assume that $f_n(\mathcal{O})\subseteq V$ for every $n\in\w$. But this means that $\delta(x)\in[S;V]$ for every $x\in\mathcal{O}$. Thus the canonical map $\delta$ is continuous at $x_0$.
\smallskip

(ii)$\Rightarrow$(i) Assume that $\delta$ is continuous. To show that $X$ is sequentially $Y$-Ascoli, we have to check that every convergent sequence $S=\{ f_n:n\in\w\}\subseteq \CC(X,Y)$  with $f_n\to f_0$  is evenly continuous at $(f_0,x)$ with $x\in X$. So fix an $x\in X$ and an open neighborhood $\mathcal{O}_{f_0(x)}\subseteq Y$ of $f_0(x)$.
Using the regularity of $Y$, choose an open neighborhood $\widetilde{\mathcal{O}}_{f_0(x)}$ of $f_0(x)$ such that $\mathrm{cl}_Y \big( \widetilde{\mathcal{O}}_{f_0(x)}\big) \subseteq \mathcal{O}_{f_0(x)}$. Set $U_0 := \big[ \{ x\}; \widetilde{\mathcal{O}}_{f_0(x)}\big] \cap S$ and $S_0:= \mathrm{cl}_S (U_0)$. Then $S_0$ is a convergent sequence such that $S\setminus S_0$ is finite and
\[
\delta(x)(g)=g(x) \in \mathrm{cl}_Y \big( \widetilde{\mathcal{O}}_{f_0(x)}\big) \subseteq \mathcal{O}_{f_0(x)}, \; \mbox{ for every } g\in S_0,
\]
and therefore $\delta(x) \in [S_0; \mathcal{O}_{f_0(x)}]$. Since $\delta$ is continuous, there is a neighborhood $O_x$ of $x$ such that $\delta(O_x)\subseteq \big[S_0; \mathcal{O}_{f_0(x)}\big]$. So, for every $t\in O_x$ and each $g\in U_0 \subseteq S_0$, we have $g(t)=\delta(t)(g)\in \mathcal{O}_{f_0(x)}$, which means that $S$ is evenly continuous at $(f_0,x)$.
\smallskip

(ii)$\Rightarrow$(iii) Fix two distinct points $y,z\in Y$, and choose an open neighborhood $V$ of $y$ such that $z\not\in V$. Since $X$ is $Y$-Tychonoff, the map $\delta$ is injective. So it remains to show that the map $\delta^{-1}: \delta(X) \to X$ is continuous. Fix an $x_0\in X$ and an open neighborhood $U$ of $x_0$. Since $X$ is $Y$-Tychonoff, there is an $f\in C(X,Y)$ such that
\begin{equation} \label{equ:seq-Ascoli-1}
f(x_0)=y \; \mbox{ and }\; f(X\SM U)\subseteq \{z\}.
\end{equation}
Set $S:=\{f_n\}_{n\in\w}$, where $f_n=f$ for every $n\in\w$, and consider the following open basic neighborhood of $\delta(x_0)$:
\[
W:= W[\delta(x_0);S,V] \subseteq C_s\big(\CC(X,Y),Y\big).
\]
Then $x\in \delta^{-1}\big(W\cap \delta(X)\big)$ if and only if  $\delta(x)\in W[\delta(x_0);S,V]$ if and only if $\delta(x)(f)=f(x)\in V$. Now, since $z\not\in V$, (\ref{equ:seq-Ascoli-1}) implies that  $x\in U$. Thus $\delta^{-1}|_{\delta(X)}$ is continuous.
\smallskip

The implications (iii)$\Rightarrow$(ii) and (i)$\Rightarrow$(iv) are trivial.
\smallskip

(iv)$\Rightarrow$(i) Assume that each point $x\in X$ is contained in a dense sequentially $Y$-Ascoli subspace of $X$. We have to show that $X$ is sequentially $Y$-Ascoli.  Take a convergent sequence $S=\{ f_n:n\in\w\}\subseteq \CC(X,Y)$  with $f_n\to f_0$. To show that $S$ is evenly continuous, fix a point $x_0\in X$ and an open neighborhood $O=O_{f_0(x_0)}$ of $f_0(x_0)$. By our assumption, the point $x_0$ is contained in a dense sequentially $Y$-Ascoli subspace $Z\subseteq X$. The density of $Z$ in $X$ implies that the restriction operator
\[
\zeta:\CC(X,Y)\to \CC(Z,Y), \quad \zeta:g\mapsto g|_Z,
\]
is injective. Since the space $Z$ is sequentially $Y$-Ascoli, there is a neighborhood $W_0\subseteq Z$ of $x_0$ and $N\in\w$ such that
\begin{equation} \label{equ:seq-Ascoli-2}
\zeta(f_n)(x) \in O \; \mbox{ for every } x\in W_0 \mbox{ and each } n\geq N.
\end{equation}
The closure $\overline{W_0}$ of $W_0$ in $X$ is a (closed) neighborhood of $z$ in $X$, and (\ref{equ:seq-Ascoli-2})  implies
\[
f_n(x) \in \overline{O} \; \mbox{ for every } x\in \overline{W_0} \mbox{ and each } n\geq N.
\]
Taking into account that $x_0$ and  $O$ were arbitrary and $Y$ is Tychonoff, we obtain that $S$ is evenly continuous.
\smallskip

(v)$\Rightarrow$(i) follows from the easy fact that every equicontinuous subset of $\CC(X,Y)$ is evenly continuous.

(i)$\Rightarrow$(v)
Let $\KK$ be a convergent sequence in $\CC(X,Y)$. So $\KK$ is evenly continuous. Fix a point $x\in X$ and an open neighborhood $U$ of zero in $Y$. Choose an open neighborhood $V$ of $0\in Y$ such that $V+V\subseteq U$. For every $f\in\KK$, choose an open neighborhood $O_f\subseteq X$ of $x$ and an open neighborhood $U_f\subseteq \KK$ of $f$ such that
\begin{equation} \label{equ:even-cont-1}
g(t)\in f(x)+V \; \mbox{ for every } \; t\in O_f \; \mbox{ and } \; g\in U_f.
\end{equation}
Since $\KK$ is compact, there are $f_1,\dots,f_n\in\KK$ such that $\KK=\bigcup_{i=1}^n U_{f_i}$. Set $O:=\bigcap_{i=1}^n O_{f_i}$. Now if $t\in O$ and $g\in\KK$, choose $f_i$ such that $g\in  U_{f_i}$ and then
\[
g(t)-g(x)=\big( g(t)-f_i(x)\big) + \big( f_i(x)-g(x)\big)\stackrel{(\ref{equ:even-cont-1})}{\in} V+V\subseteq U.
\]
Thus $\KK$ is equicontinuous. \qed
\end{proof}

Let $X$ be a topological space. A family  $\{ A_i\}_{i\in I}$ of subsets of  $X$ is called {\em compact-finite} if for every compact $K\subseteq X$, the set $\{ i\in I: K\cap A_i\}$ is finite; and $\{ A_i\}_{i\in I}$ is called {\em locally finite} if for every point $x\in X$ there is a neighborhood $X$ of $x$ such that the set $\{ i\in I: U\cap A_i\}$ is finite. A family  $\{ B_i\}_{i\in I}$ of subsets of  $X$ is called {\em strongly} ({\em functionally}) {\em compact-finite} if for every $i\in I$, there is an (functionally) open neighborhood  $U_i$ of $B_i$ such that the family $\{ U_i\}_{i\in I}$ is compact-finite.

\begin{proposition} \label{p:seq-Ascoli-1}
Let $Y$ be an abelain topological group containing at least two points, and let $X$ be a   $Y$-Tychonoff space. Consider the following assertions:
\begin{enumerate}
\item[{\rm (i)}] every strongly compact-finite sequence of closed subsets of $X$ is locally finite;
\item[{\rm (ii)}] every strongly  functionally  compact-finite sequence of functionally closed subsets of $X$ is locally finite;
\item[{\rm (iii)}] $X$ is a sequentially $Y$-Ascoli space.
\end{enumerate}
Then {\rm (i)$\Rightarrow$(ii)$\Rightarrow$(iii)}. If in addition $X$ is $Y$-normal, then {\rm (iii)$\Rightarrow$(i)}.
\end{proposition}

\begin{proof}
The implication (i)$\Rightarrow$(ii) is trivial.

(ii)$\Rightarrow$(iii) Suppose for a contradiction that $X$ is not sequentially $Y$-Ascoli. Then, by the equivalence (i)$\Leftrightarrow$(v) of Theorem \ref{t:seq-Ascoli}, there exists a null-sequence $\{ f_n:n\in\w\}$ in $\CC(X,Y)$ which is not equicontinuous at some point $x_0\in X$. So there is a functionally open neighborhood $W$ of zero in $Y$ such that the closed sets
\[
A_n := \{ x \in X: f_n(x)-f_n(x_0) \in Y\SM W\} \quad (n\in\w)
\]
satisfy the condition
$
x_0\in \overline{\bigcup_{n\in\w} A_n} \setminus \bigcup_{n\in\w} A_n,
$
i.e. the family $\{ A_n:n\in\w\}$ is not locally finite. Observe also that the sets $A_n$ are functionally closed in $X$. Choose a functionally open neighborhood $W_0\subseteq Y$ of zero such that $\overline{W_0} +\overline{W_0} \subseteq W$. For every $n\in\w$, set
\[
U_n:=\{ x \in X: f_n(x)-f_n(x_0)\in  Y\SM \overline{W_0}\},
\]
so  $U_n$ is a functionally open neighborhood of $A_n$. Hence to get a contradiction it remains to show that the family $\UU=\{ U_n:n\in\w\}$ is compact-finite. Let $K$ be a compact subset of $X$. Since $f_n\to \mathbf{0}$ in the compact-open topology, there exists an $N\in \w$ such that $f_n\in [K\cup\{x_0\};W_1]$ for every $n\geq N$, where the open symmetric neighborhood $W_1\subseteq Y$ of $0$ is chosen so that $\overline{W_1} +\overline{W_1} \subseteq W_0$. Therefore, for every $n\geq N$ and each $x\in K$, we have
$
f_n(x)-f_n(x_0)\in W_1 +W_1 \subseteq W_0
$
and hence $U_n \cap K=\emptyset$. Thus $\UU$ is compact-finite.

(iii)$\Rightarrow$(i) Assume that $X$ is a $Y$-normal space. Fix two distinct points $y,z\in Y$, and choose an open neighborhood $V$ of $y$ such that $z\not\in V$. Suppose for a contradiction that there exists a  strongly compact-finite sequence $\{ A_n:n\in\w\}$ of closed subsets  in $X$ which is not locally finite. For every $n\in\w$, choose an  open neighborhood $U_n$ of $A_n$ such that the sequence $\UU=\{ U_n:n\in\w\}$ is compact-finite. Since $X$ is $Y$-normal, for every $n\in\w$ there is a continuous function $f_n:X \to Y$ such that $f_n(A_n) =\{ z\}$ and $f_n(X\setminus U_n)\subseteq \{ y\}$. Since $\UU$ is compact-finite, we obtain that $f_n \to \yyy$ in $\CC(X,Y)$. As $\{ A_n:n\in\w\}$ is not locally-finite, there exists a point $x_0\in X$ such that,  for every neighborhood $U$ of $x_0$, the set $\{ n\in\w: A_n \cap U\not= \emptyset\}$ is infinite. Hence,  for every neighborhood $U$ of $x_0$, there is an $n\in\w$ such that $f_n(x)=z$ for some $x\in A_n$. Therefore the convergent sequence $\{ f_n:n\in\w\}\cup\{ \yyy\}$ is not evenly continuous at $(x_0,\yyy)$. Thus $X$ is not sequentially $Y$-Ascoli, a contradiction.\qed
\end{proof}

If $Y=\IR$, we have the following characterization of sequentially Ascoli spaces.

\begin{theorem}\label{t:seq-R-Ascoli}
Let $X$ be a Tychonoff space, and let $\delta: X\to C_s\big(\CC(X)\big)$ be the canonical map. Then the following assertions are equivalent:
\begin{enumerate}
\item[{\rm (i)}] $X$ is a sequentially $\IR$-Ascoli space;
\item[{\rm (ii)}] $\delta$ is continuous;
\item[{\rm (iii)}]  $\delta$ is an embedding;
\item[{\rm (iv)}] each point $x\in X$ is contained in a dense sequentially Ascoli subspace of $X$;
\item[{\rm (v)}] every convergent sequence in $\CC(X)$ is equicontinuous, i.e. $X$ is sequentially Ascoli;
\item[{\rm (vi)}] every strongly  functionally  compact-finite sequence of functionally closed subsets of $X$ is locally finite.
\end{enumerate}
\end{theorem}

\begin{proof}
Since a topological space is Tychonoff if and only if it is $\IR$-Tychonoff, the equivalences (i)-(v) follow from Theorem \ref{t:seq-Ascoli}. The implication (vi)$\Rightarrow$(i) follows from Proposition \ref{p:seq-Ascoli-1}.

(v)$\Rightarrow$(vi)  
Suppose for a contradiction that there exists a  strongly  functionally  compact-finite sequence of functionally closed subsets $\{ A_n:n\in\w\}$ in $X$ which is not locally finite. For every $n\in\w$, choose a functionally  open neighborhood $U_n$ of $A_n$ such that the sequence $\UU=\{ U_n:n\in\w\}$ is compact-finite. Since $X$ is Tychonoff,  Theorem 1.5.13 of \cite{Eng} implies that for every $n\in\w$ there is a continuous function $f_n:X \to [0,1]$ such that $f_n(A_n) =\{ 1\}$ and $f_n(X\setminus U_n) =\{ 0\}$. Since $\UU$ is compact-finite, we obtain that $f_n \to \mathbf{0}$ in $\CC(X)$. As $\{ A_n:n\in\w\}$ is not locally-finite, there exists a point $x_0\in X$ such that,  for every neighborhood $U$ of $x_0$, the set $\{ n\in\w: A_n \cap U\not= \emptyset\}$ is infinite. Hence,  for every neighborhood $U$ of $x_0$, there is an $n\in\w$ such that $f_n(x)=1$ for some $x\in A_n$. Therefore the convergent sequence $\{ f_n:n\in\w\}\cup\{ \mathbf{0}\}$ is not equicontinuous at the point $x_0$. \qed
\end{proof}

Now we establish some elementary properties of sequentially $Y$-Ascoli spaces. Recall that a subspace $Z$ of a topological space $X$ is a {\em retract of $X$} if there is a continuous map $r:X\to Z$ such that $r(z)=z$ for all $z\in Z$. The proof of the next proposition is actually a partial case (when a compact set is a convergent sequence) of Proposition 5.2 of \cite{BG} and hence is omitted.

\begin{proposition} \label{p:seq-AscoliCat}
Let $Y$ be a Tychonoff space.
\begin{enumerate}
\item[{\rm (i)}] If $X$ is a sequentially $Y$-Ascoli space, then each retract $Z$ of $X$ is sequentially $Y$-Ascoli.
\item[{\rm (ii)}] For any family $\{  X_i\}_{i\in I}$ of sequentially $Y$-Ascoli spaces the topological sum $X:=\bigoplus_{i\in I} X_i$ is sequentially $Y$-Ascoli.
\item[{\rm (iii)}] Each sequentially $Y$-Ascoli space $X$ is sequentially $Y^\kappa$-Ascoli for every cardinal $\kappa$.
\item[{\rm (iv)}] Each sequentially $Y$-Ascoli space $X$ is sequentially $Z$-Ascoli for every subspace $Z\subseteq Y$.
\end{enumerate}
\end{proposition}

Since any (zero-dimensional) Tychonoff space embeds into some power $\mathbb{R}^\kappa$ (respectively, $2^\kappa$), Theorem \ref{t:seq-R-Ascoli} and Proposition \ref{p:seq-AscoliCat} imply
\begin{corollary}\label{l:seq-AscoliCat}
\begin{enumerate}
\item[{\rm (i)}] If $X$ is a sequentially Ascoli space, then $X$ is sequentially $Y$-Ascoli for every Tychonoff space $Y$.
\item[{\rm (ii)}] If $X$ is a  sequentially  $2$-Ascoli space, then $X$ is  sequentially  $Y$-Ascoli for every zero-dimensional $T_1$-space $Y$.
\end{enumerate}
\end{corollary}

In what follows we shall use repeatedly the following  proposition.
\begin{proposition}[\cite{Gabr:weak-bar-L(X),GKP}] \label{p:Ascoli-sufficient}
Assume that for a set $I$ (respectively, $I=\w$), the Tychonoff space  $X$ admits a  family $\U =\{ U_i : i\in I\}$ of open sets, a subset $A=\{ a_i : i\in I\} \subseteq X$ and a point $z\in X$ such that
\begin{enumerate}
\item[{\rm (i)}] $a_i\in U_i$ for every $i\in I$;
\item[{\rm (ii)}] $\big|\{ i\in I: C\cap U_i\not=\emptyset \}\big| <\infty$  for each compact subset $C$ of $X$;
\item[{\rm (iii)}] $z$ is a cluster point of $A$.
\end{enumerate}
Then $X$ is not a (sequentially) Ascoli space.
\end{proposition}

Analogously we have the following result. The characteristic function of a subset $A$ of a set $\Omega$ is denoted by $\mathbf{1}_A$.
\begin{proposition} \label{p:2-Ascoli-sufficient}
Assume that for a set $I$ (respectively, $I=\w$), the Tychonoff space  $X$ admits a  family $\U =\{ U_i : i\in I\}$ of clopen sets and a point $z\in X$ such that
\begin{enumerate}
\item[{\rm (i)}] $\big|\{ i\in I: C\cap U_i\not=\emptyset \}\big| <\infty$  for each compact subset $C$ of $X$;
\item[{\rm (ii)}] $z$ is a cluster point of $\bigcup_{i\in I} U_i$.
\end{enumerate}
Then $X$ is not a (sequentially) $2$-Ascoli space.
\end{proposition}

\begin{proof}
We have to find a compact subset (or a convergent sequence) $K$ in $\CC(X,\mathbf{2})$ which is not equicontinuous. Since all $U_i$ are clopen, for every $i\in I$, the characteristic function $\mathbf{1}_{U_i}$ is continuous. Now we define $K:=\{ \mathbf{1}_{U_i}:i\in I\}\cup\{\mathbf{0}\}$. By (ii), the family $K$ is not equicontinuous at the point $z$. So, to show that $X$ is not a (sequentially) $2$-Ascoli space, it suffices to prove that $K\subseteq \CC(X,\mathbf{2})$ is compact (or converges to $\mathbf{0}$, respectively). Fix an arbitrary standard neighborhood $[C;0]$ of $\mathbf{0}\in \CC(X,\mathbf{2})$. Then, by (i), the set $J:=\{ i\in I: C\cap U_i\not=\emptyset \}$ is finite. Therefore, for every $i\not\in  J$ we have $\mathbf{1}_{U_i}\in [C;0]$. Thus $K$ is compact (or converges to $\mathbf{0}$). \qed
\end{proof}

\begin{example} \label{exa:subspace-k-FU-non-Ascoli} {\em
A closed subspace of a $\kappa$-Fr\'{e}chet--Urysohn space can be not a sequentially Ascoli space. Indeed, let $A$ be a countable subset on the unit sphere of $\ell_1$ constructed in Proposition 4.1 of \cite{GKP}. It immediately follows from \cite[Proposition~4.1]{GKP} and Proposition \ref{p:Ascoli-sufficient} that the space  $A$ is not sequentially Ascoli. On the other hand, the space $A$ being Lindel\"{o}f is homeomorphic to a closed subspace of some product $\IR^\lambda$. It remains to note that  $\IR^\lambda$ is a $\kappa$-Fr\'{e}chet--Urysohn space by Corollary 2.3 of \cite{Gabr-B1}.\qed}
\end{example}

\begin{example} \label{exa:product-non-Ascoli} {\em
A product of a Polish space and a sequential $k_\w$-space can be not sequentially Ascoli. Indeed, let $\phi$ be the inductive limit of finite-dimensional vector spaces. It is well known that the space $\phi$ is a sequential non-Fr\'{e}chet--Urysohn $k_\w$-space (see for example Proposition 2.1 in \cite{Gab-LF}). Taking into account Proposition  \ref{p:Ascoli-sufficient}, in Theorem 1.2 of \cite{Gab-LF} it is proved that the product $\ell_2\times \phi$ is not a sequentially Ascoli space. In particular, the three space property does not hold in the class of (sequentially) Ascoli spaces.\qed}
\end{example}

Gruenhage and Tanaka proved that the square $V(\w_1)\times V(\w_1)$ of the Fr\'{e}chet--Urysohn fan has uncountable tightness and hence is not a $k$-space,  see \cite[Lemma~7.6.22]{ArT}. Below we show that $V(\w_1)\times V(\w_1)$ is not a $2$-Ascoli space, and hence the square of a Fr\'{e}chet--Urysohn space may even not be $2$-Ascoli.

\begin{proposition} \label{p:square-non-Ascoli}
The product $V(\w_1)\times V(\w_1)$ is not a $2$-Ascoli space.
\end{proposition}

\begin{proof}
We shall use the idea of Gruenhage and Tanaka and consider two families $\{ A_\alpha: \alpha\in\w_1\}$ and $\{ B_\alpha: \alpha\in\w_1\}$ of infinite subsets of $\w$ satisfying the following two conditions:
\begin{enumerate}
\item[{\rm (a)}] $A_\alpha \cap B_\beta$ is finite for all $\alpha,\beta \in\w_1$,
\item[{\rm (b)}] for no $D\subseteq \w$, all sets $A_\alpha\SM D$ and $B_\alpha\cap D$, $\alpha\in\w_1$, are finite.
\end{enumerate}
Set $z:= (x_\infty,x_\infty)$. For all $\alpha,\beta \in\w_1$ and $n\in A_\alpha \cap B_\beta$, we put
$
U_{n,\alpha,\beta} :=\{ (x_{n,\alpha},x_{n,\beta})\}
$
(so $U_{n,\alpha,\beta}$ is a clopen subset of $V(\w_1)\times V(\w_1)$).
To apply Proposition \ref{p:2-Ascoli-sufficient} to the family $\U=\{U_{n,\alpha,\beta}\}$ and the point $z$ we have to check conditions (i)-(ii). The condition (ii) is proved in Lemma~7.6.22 of \cite{ArT}.

To check (i), let $C$ be a compact subset of $V(\w_1)\times V(\w_1)$. Then there is a finite subset $\Gamma$ of $\w_1$ such that $C$ is contained in the product
$
\Big( \bigcup_{\gamma\in\Gamma} S_\gamma\Big) \times \Big( \bigcup_{\gamma\in\Gamma} S_\gamma\Big).
$
Therefore, if $(x_{n,\alpha},x_{n,\beta})\in C$, then $\alpha,\beta\in \Gamma$ and, by construction, $n\in A_\alpha \cap B_\beta$. Now the finiteness of the sets $\Gamma$ and $A_\alpha \cap B_\beta$ (see (a)) imply that the set
\[
\{ (n,\alpha,\beta): C\cap U_{n,\alpha,\beta}\not=\emptyset\}= \{ (n,\alpha,\beta): (x_{n,\alpha},x_{n,\beta})\in C\}
\]
is finite, and hence (i) of Proposition \ref{p:2-Ascoli-sufficient} holds true. Thus $V(\w_1)\times V(\w_1)$ is not $2$-Ascoli.\qed
\end{proof}

We do not know whether  the product $V(\w_1)\times V(\w_1)$ is not sequentially Ascoli.

\begin{problem}
Is there a (sequentially) Ascoli group whose square is not (sequentially) Ascoli?
\end{problem}

The next proposition complements Proposition \ref{p:seq-AscoliCat}.

\begin{proposition} \label{p:open-image-Ascoli}
Let $Y$ be a Tychonoff space containing at least two points,  $X$ and $Z$ be $Y$-Tychonoff spaces, and let $p:X\to Z$ be an open surjective map. If $X$ is a (sequentially) $Y$-Ascoli space, then so is $Z$.
\end{proposition}

\begin{proof}
Denote by $p^\ast: \CC(Z,Y)\to \CC(X,Y)$, $p^\ast(f):=f\circ p$, the adjoint map of $p$. Let $K$ be a compact subset (or a convergent sequence) of $\CC(Z,Y)$. Fix $f\in K$, $s\in Z$ and a neighborhood $O_{f(s)}\subseteq Y$ of $f(s)$. Choose $t\in X$ such that $p(t)=s$. Since $p^\ast(K)$ is a compact subset (or a convergent sequence) of $\CC(X,Y)$, there are  open neighborhoods $U\subseteq X$ of $t$ and $U_{p^\ast(f)}$ of $p^\ast(f)$ in $\CC(X,Y)$ such that
\begin{equation} \label{equ:Ascoli-open}
g(x)\in O_{f(s)} \; \mbox{ for every } \; x\in U \; \mbox{ and }\; g\in U_{p^\ast(f)} \cap p^\ast(K).
\end{equation}
As $p$ is open, $p(U)$ is an open neighborhood of $s$. Set $U_f := (p^\ast)^{-1}(U_{p^\ast(f)})\cap K$, so $U_f$ is a neighborhood of $f$ in $K$. Now, for every $z=p(x)\in p(U)$ with $x\in U$ and each $F\in U_{f}$, (\ref{equ:Ascoli-open}) implies
\[
F(z)=F\circ p(x)=p^\ast(F)(x)\in O_{f(s)} .
\]
Thus $K$ is evenly continuous. \qed
\end{proof}

Since quotient maps of topological groups are open (see \cite[Theorem~5.26]{HR1}) Proposition \ref{p:open-image-Ascoli} implies
\begin{corollary} \label{c:open-image-Ascoli}
Let $Y$ be a Tychonoff space containing at least two points, and let $H$ be a normal closed subgroup of a $Y$-Tychonoff group $G$ such that the quotient group $G/H$ is $Y$-Tychonoff. If $G$ is (sequentially) $Y$-Ascoli, then so is $G/H$. In particular, a Hausdorff quotient group of a (sequentially) Ascoli group is a (sequentially) Ascoli group.
\end{corollary}

Let us recall that a topological space $X$ is called a {\em $P$-space} if each $G_\delta$-set in $X$ is open.
In Proposition 2.9 of \cite{Gabr:weak-bar-L(X)} we proved that any non-discrete Tychonoff $P$-space is  sequentially Ascoli but not Ascoli. Below we extend this result.
\begin{proposition} \label{p:P-space-2-Ascoli}
Let $X$ be a non-discrete Tychonoff $P$-space. Then:
\begin{enumerate}
\item[{\rm (i)}] $X$ is a hereditary sequentially Ascoli space;
\item[{\rm (ii)}] $X$ is not $2$-Ascoli;
\item[{\rm (iii)}] if additionally $X$ is a one-point Lindel\"{o}fication of an uncountable discrete set, then $X$ is a hereditary  Baire space.
\end{enumerate}
\end{proposition}

\begin{proof}
(i) Let $Z$ be a subspace of $X$. Then $Z$ is also a Tychonoff $P$-space. Hence, by Proposition 2.9 of \cite{Gabr:weak-bar-L(X)}, $Z$ is a sequentially Ascoli space. Thus $X$ is a hereditary sequentially Ascoli space.

(ii) Every Tychonoff $P$-space is zero-dimensional and every compact subset of $X$ is finite (see Lemma 2.8 in \cite{Gabr:weak-bar-L(X)}). Thus $X$ is not $2$-Ascoli by Proposition 5.12 of \cite{BG}.

(iii)  Let $Z$ be a subspace of $X$.  If $Z$ does not contain the unique non-isolated point $x_\infty$ of $X$ or is countable, then $Z$ with the induced topology is discrete and hence Baire. If $Z$ is uncountable and contains $x_\infty$, then $Z$ is topologically isomorphic to a one-point Lindel\"{o}fication of the uncountable discrete space $Z\SM\{x_\infty\}$, and hence $Z$ is homeomorphic to a one point Lindel\"{o}fication of some uncountable discrete set. So it is sufficient to show that $X$ is Baire. It is easy to see that any open dense subset of $X$ is either $X$ or $X\SM \{x_\infty\}$. Therefore any intersection of dense open subsets of $X$ contains the dense subset $X\SM \{x_\infty\}$. Thus $X$ is Baire. \qed
\end{proof}

Proposition \ref{p:P-space-2-Ascoli} motivates the following problem.
\begin{problem} \label{prob:2-Ascoli-seq-Ascoli}
Is every $2$-Ascoli zero-dimensional space $X$  sequentially Ascoli?
\end{problem}

If $X$ is $2$-normal, Problem \ref{prob:2-Ascoli-seq-Ascoli} has a positive answer even in a stronger form.
\begin{proposition} \label{p:2-Ascoli-seq-Ascoli}
Let $X$ be a zero-dimensional space. If $X$ is $2$-normal, then $X$ is a sequentially $2$-Ascoli space if and only if it is sequentially Ascoli.
\end{proposition}

\begin{proof}
Every sequentially Ascoli space is sequentially $2$-Ascoli. Conversely, assume that $X$ is a sequentially $2$-Ascoli space. To prove that $X$ is a sequentially Ascoli space, let $\{ f_n\}_{n\in\w}$ be a null-sequence in $\CC(X)$ and suppose for a contradiction that it is not equicontinuous at some point $x_0\in X$. Without loss of generality we can assume that $f_n(x_0)=0$ for all $n\in\w$ replacing $f_n$ by $f_n-f_n(x_0)$ if needed. Then there is $\delta>0$ such that the  sets
$A_n:= \{ x\in X: |f_n(x)|\geq 2\delta\}$  $(n\in\w)$ satisfy the following condition
\begin{equation} \label{equ:2-seq-Ascoli-1}
x_0 \in \overline{\bigcup_{n\in\w}A_n }.
\end{equation}
For every $n\in\w$, define $B_n:= \{ x\in X:  |f_n(x)|\leq \delta\}$. Then $A_n$ and $B_n$ are closed and disjoint. Since $X$ is $2$-normal, there is a function $g_n\in C(X,\mathbf{2})$ such that
\begin{equation} \label{equ:2-seq-Ascoli-2}
g_n(B_n)=\{0\} \;\; \mbox{ and }\;\; g_n(A_n)\subseteq\{1\} \quad (n\in\w).
\end{equation}
Now (\ref{equ:2-seq-Ascoli-1}) and (\ref{equ:2-seq-Ascoli-2}) imply that the sequence $\{ g_n\}_{n\in\w}$ is not equicontinuous at $x_0$. Therefore, by Theorem \ref{t:seq-Ascoli}(v), to get a desired contradiction it suffices to show that $g_n\to \mathbf{0}$ in $\CC(X,\mathbf{2})$. To this end, fix an arbitrary compact subset $K$ of $X$. Since $f_n\to \mathbf{0}$ in $\CC(X)$, choose an $m\in\w$ such that
\[
f_n\in W[\mathbf{0}; K, (-\delta,\delta)] \;\; \mbox{ for every }\; n\geq m.
\]
Then, for every $n\geq m$, we have $K\subseteq \{ x\in X:  |f_n(x)|<\delta\} \subseteq B_n$ and hence $g_n(K)=\{0\}$. Thus $g_n\to \mathbf{0}$ in $\CC(X,\mathbf{2})$.\qed
\end{proof}

Arhangel'skii proved (see \cite[3.12.15]{Eng}) that if $X$ is a hereditary $k$-space (i.e., {\em every} subspace of $X$ is a $k$-space), then $X$ is Fr\'{e}chet--Urysohn. Below we essentially strengthen this remarkable result by proving the following theorem.

\begin{theorem} \label{t:her-Ascoli-k-space}
A Tychonoff space is hereditary $2$-Ascoli if and only if it is Fr\'{e}chet--Urysohn. Consequently, a  Tychonoff space is hereditary Ascoli if and only if it is Fr\'{e}chet--Urysohn.
\end{theorem}

\begin{proof}
If we shall show that every hereditary $2$-Ascoli space is a $k$-space, we shall prove that it is a hereditary $k$-space, and applying
Arhangel'skii's theorem \cite[3.12.15]{Eng} 
we obtain that it is Fr\'{e}chet--Urysohn.
So, let $X$ be a hereditary $2$-Ascoli space and suppose for a contradiction that $X$ is not a $k$-space. Then there exists a non-closed subspace $A$ of $X$ such that $A\cap K$ is closed in $K$ for every compact subspace $K$ of $X$. Choose an arbitrary point $z\in \overline{A}\SM A$ and consider the subspace $Y:= A\cup\{ z\}$ of $X$.

We claim that if $K$ is a compact subspace of $Y$ then either $z\not\in K$ or $z$ is an isolated point of $K$. Indeed, if $z$ is a non-isolated point of $K$, then $z \in \overline{A\cap K}$. Since $z\not\in A$ we obtain that $A\cap K$ is not closed in $K$ that contradicts the choice of $A$.

Consider the family $\U$ of all collections $\{ U_i\}_{i\in I}$ of pairwise disjoint open sets in $Y$ such that $z\not\in \bigcup_{i\in I} \overline{U_i}$. By Zorn's lemma, the family $\U$ has a maximal (under inclusion) collection $\{ V_i\}_{i\in I}$. The maximality of $\{ V_i\}_{i\in I}$ and the fact that $z$ is not isolated in $Y$ imply that $z\in \overline{\bigcup_{i\in I} V_i}$.

Consider the subspace $Z:= \{ z\} \cup \bigcup_{i\in I} V_i$ of $Y$ and hence of $X$. Since $X$ is hereditary $2$-Ascoli, the space $Z$ is a $2$-Ascoli space. Observe that $z$ is a non-isolated point of $Z$, and the fact $z\not\in \bigcup_{i\in I} \overline{V_i}$ implies that all subspaces $V_i$ are clopen in $Z$. Therefore, for every $i\in I$, the function $\mathbf{1}_{V_i}$ is continuous on $Z$. Set $\KK:= \{\mathbf{0}\}\cup \{ \mathbf{1}_{V_i}\}_{i\in I}$. Let us show that $\KK$ is a compact subset of $\CC(Z,\mathbf{2})$. Indeed, let $[K;0]$ be a standard open neighborhood of the zero function $\mathbf{0}\in \CC(Z,\mathbf{2})$, where $K$ is a compact subset of $Z$. By the claim we can assume that $K= \{z\} \cup K_0$, where $K_0$ is a compact subset of  $\bigcup_{i\in I} V_i$ and $z\not\in K_0$. Then the set $J:=\{i\in I: V_i\cap K_0 \not=\emptyset\}$ is finite. Clearly, if $i\not\in J$, then $\mathbf{1}_{V_i}\in [K;0]$. Thus $\KK$ is compact.
Since $z\in \overline{\bigcup_{i\in I} V_i}$, it follows that the compact set $\KK$ is not equicontinuous at the point $z$. Therefore the space $Z$ is not $2$-Ascoli, a contradiction. Thus the space $X$ must be a $k$-space. \qed
\end{proof}

We end this section with the question to find natural classes of sequentially Ascoli spaces $X$ which are Ascoli.

\begin{problem} \label{prob:seq-Ascoli-1}
Let $X$ be a Tychonoff space such that every compact subset of $\CC(X)$ is metrizable (for example, $X$ is a cosmic space or more generally $X$ contains a dense $\sigma$-compact subspace). Is it true that $X$ is sequentially Ascoli if and only if it is an Ascoli space?
\end{problem}

The condition for the compact subsets of $\CC(X)$ of being metrizable  in Problem \ref{prob:seq-Ascoli-1} is essential (also one cannot replace the condition to have a dense $\sigma$-compact subspace by Lindel\"{o}fness of $X$  as Proposition \ref{p:P-space-2-Ascoli} shows).
\begin{example} \label{exa:Ascoli-seq-Ascoli}
There is a sequentially Ascoli non-Ascoli space $X$ such that every compact subset of $\CC(X)$ is  Fr\'{e}chet--Urysohn.
\end{example}

\begin{proof}
Let $X$ be a non-discrete Lindel\"{o}f $P$-space. Then, by Proposition \ref{p:P-space-2-Ascoli}, $X$ is a sequentially Ascoli non-Ascoli space. On the other hand, since every compact subset of a $P$-space is finite we have $\CC(X)=C_p(X)$. Therefore, by Theorem II.7.16 of \cite{Arhangel}, the space $\CC(X)$ and hence all its compact subspaces are Fr\'{e}chet--Urysohn. \qed
\end{proof}



\section{Sequentially Ascoli spaces in some compact-type classes of Tychonoff spaces} \label{sec:compact-Ascoli}



Let $G$ be a locally compact abelian (lca for short) group. The group $G$ endowed with the Borh topology $\tau_b$ induced from the Bohr compactification $bG$ of $G$ is denoted by $G^+$. Numerous topological and algebraic-topological properties of the precompact group $G^+$ are well-studied, we refer the reader to Chapter 9 of \cite{ArT} and references therein. It is an easy consequence of the Glicksberg theorem (which states that  the groups $G$ and $G^+$ have the same compact sets) that the group $G^+$ is a $k$-space if and only if $G$ is compact.  Below we essentially extend this result.

\begin{theorem} \label{t:seq-Ascoli-G+}
If $G$ is an lca group, then $G^+$ is sequentially Ascoli if and only if $G$ is compact.
\end{theorem}

\begin{proof}
If $G$ is compact, then $G^+=G$ is an Ascoli space. Conversely, assume that $G^+$ is sequentially Ascoli. We have to show that $G$ is compact.  By Theorem 24.30 of \cite{HR1}, $G$ is topologically isomorphic to $\IR^n\times G_0$, where $n\in\w$ and $G_0$ is an lca group containing an open compact subgroup $G_1$.

{\em Claim 1. $n=0$ and hence $G=G_0$}. Indeed, suppose for a contradiction that $n>0$. Then $\IR^+$ is a direct summand of $G^+$ (see for example Problem 9.9.N of \cite{ArT}). Hence, by Corollary \ref{c:open-image-Ascoli}, $\IR^+$ is sequentially Ascoli. So to get a contradiction we shall show below that the space $\IR^+$ is {\em not} sequentially Ascoli. Since $\IR^+$ is Tychonoff, for every natural number $n>0$, there is a continuous function $f_n:\IR^+ \to [0,2]$ such that
\[
f_n \big( [-n,n]\big) =\{0\} \; \mbox{ and }\; f_n \big( [-3n,-n-1]\cup[n+1,3n]\big)=\{2\}.
\]
By the Glicksberg theorem, for every compact subset $K$ of $\IR^+$ there is an $n\in\w$ such that $K\subseteq [-n,n]$. This fact and the definition of the functions $f_n$ imply that $f_n\to f_0:=\mathbf{0}$ in $\CC(\IR^+)$. We show that the null-sequence $S=\{ f_n\}_{n\in\w}$ is not equicontinuous at zero $0\in \IR^+$ (which means that $\IR^+$ is  not sequentially Ascoli, see Theorem \ref{t:seq-R-Ascoli}). Indeed, let $\e=1$ and fix an arbitrary open neighborhood $U$ of $0\in \IR^+$. Since the group $\IR^+$ is not locally compact, for every $n\in\w$ the intersection
\[
U\cap \big( (-\infty,-n-1]\cup[n+1,\infty)\big)
\]
is not empty. So there are $n>0$ and $x_n\in U\cap\big([-3n,-n-1]\cup[n+1,3n]\big)$. Then $f_n(x_n)-f_n(0)=2> \e$. Thus $S$ is not  equicontinuous.

{\em Claim 2. $G=G_0$ is compact.} Indeed, suppose for a contradiction that $G$ is not compact. Then the quotient group $G/G_1$ is infinite and discrete. Applying Problem 9.9.N of \cite{ArT} (or a more general Lemma 2.3 of \cite{Gab-Top-Nul}), we obtain $(G/G_1)^+ = G^+/G_1$ and hence, by Corollary \ref{c:open-image-Ascoli},  the group $(G/G_1)^+$ is sequentially Ascoli. By 16.13(c) of \cite{HR1}, the group $G/G_1$ has a subgroup $H$ such that the discrete group $T:=(G/G_1)/H$ is countably infinite. Once again applying Problem 9.9.N of \cite{ArT} and Corollary \ref{c:open-image-Ascoli}, we obtain that the countably infinite group $T^+$ is sequentially Ascoli. By the Glicksberg theorem, every compact subset of $T^+$ is finite. Therefore, by Proposition 2.10 of \cite{Gabr:weak-bar-L(X)}, $T^+$ is not sequentially Ascoli. This contradiction shows that the group $G$ must be compact.\qed
\end{proof}

Every Hausdorff compact space is Ascoli. So it is natural to consider other classes of compact-type spaces. Let us recall that a topological space $X$ is called
\begin{enumerate}
\item[$\bullet$] {\em sequentially compact} if every sequence in $X$ has a convergent subsequence;
\item[$\bullet$] {\em $\w$-bounded} if every sequence in $X$ has compact closure;
\item[$\bullet$] {\em totally countably compact} if every sequence in $X$ has a subsequence with compact closure;
\item[$\bullet$]  {\em near sequentially compact } if for any sequence $(U_n)_{n\in\w}$ of open sets in $X$ there exists a sequence $(x_n)_{n\in\w}\in\prod_{n\in\w}U_n$ containing a convergent subsequence $(x_{n_k})_{k\in\w}$;
\item[$\bullet$] {\em countably compact} if every sequence in $X$ has a cluster point.
\end{enumerate}
Near sequentially compact spaces  were introduced and  studied by Dorantes-Aldama and Shakhmatov \cite{DAS1} as selectively sequentially pseudocompact spaces. However we shall use the term ``near sequentially compact spaces'' since it is more self-suggesting and is a partial case of a more general notion of ``near sequentially compact subset'' introduced in \cite{BG-JNP}.
In the next diagram we summarize the relationships between the above-defined notions and note that none of these implications is reversible (see \cite{DAS1} and \cite[Section~2]{Vaughan}):
\[
\xymatrix{
\mbox{compact} \ar@{=>}[r] & \mbox{$\w$-bounded}  \ar@{=>}[r] & {\substack{\mbox{totally countably} \\ \mbox{compact}}}  \ar@{=>}[r] & {\substack{\mbox{countably} \\ \mbox{compact}}}  \ar@{=>}[d]\\
& {\substack{\mbox{sequentially} \\ \mbox{compact}}}  \ar@{=>}[r]  \ar@{=>}[ru] &  {\substack{\mbox{near sequentially} \\ \mbox{compact}}} \ar@{=>}[r]  & \mbox{pseudocompact}
}
\]


\begin{theorem} \label{t:seq-compact-seq-Ascoli}
Let $X$ be a Tychonoff space satisfying at least one of the next conditions:
\begin{enumerate}
\item[{\rm(i)}] $X$ is totally countably compact;
\item[{\rm(ii)}] $X$ is near sequentially compact.
\end{enumerate}
Then $X$ is a sequentially Asoli space.
\end{theorem}

\begin{proof}
Suppose for a contradiction that $X$ is not sequentially Ascoli. Then, by the equivalence (i)$\Leftrightarrow$(v) of Theorem \ref{t:seq-Ascoli}, there exists a null-sequence $\{ f_n:n\in\w\}$ in $\CC(X)$ which is not equicontinuous at some point $x_0\in X$. Replacing $f_n$ by $f_n -f_n(x_0)$ if needed we can assume that $f_n(x_0)=0$ for every $n\in\w$. Then there is an $\e>0$ such that the closed sets
\[
A_n := \{ x \in X: f_n(x) \in (-\infty, -\e]\cup[\e,\infty)\} \quad (n\in\w)
\]
and the open sets
\[
B_n := \{ x \in X: f_n(x) \in (-\infty, -\e)\cup(\e,\infty)\} \quad (n\in\w)
\]
satisfy the condition $x_0\in \overline{\bigcup_{n\in\w} B_n} \setminus \bigcup_{n\in\w} A_n$. Passing to a subsequence if needed we can assume also that all sets $B_n$s are nonempty. Let us show that the family $\AAA=\{ A_n:n\in\w\}$ is compact-finite. Indeed, let $K$ be a compact subset of $X$. Since $f_n\to \mathbf{0}$ in the compact-open topology, there exists an $N\in \w$ such that $f_n\in [K;\e]$ for every $n\geq N$. Therefore, $A_n\cap K=\emptyset$ for every $n\geq N$. Thus $\AAA$ is compact-finite. Now we consider the cases (i) and (ii).

(i) Assume that $X$ is totally countably compact. For every $n\in\w$ choose a point $a_n\in A_n$. Since $X$ is totally countably compact, there is a subsequence $\{a_{n_k}\}_{k\in\w}$ of $\{a_{n}\}_{n\in\w}$ which has compact closure $C$. It is clear that the compact set $C$ intersects infinitely many $A_n$s, and hence $\AAA$ is not compact-finite. This is a desired contradiction.

(ii) Assume that $X$ is near sequentially compact. Since all $B_n$s are open in $X$ and $X$ is near sequentially compact, there is a sequence $(x_n)_{n\in\w}\in\prod_{n\in\w}B_n$ containing a subsequence $(x_{n_k})_{k\in\w}$ that converges to some point $x^\ast\in X$. So the convergent sequence $(x_{n_k})_{k\in\w}\cup\{ x^\ast\}$ intersects infinitely many $A_n$s, and hence $\AAA$ is not compact-finite.  This contradiction finishes the proof.\qed
\end{proof}

However, there are countably compact Tychonoff spaces $X$ which are not sequentially Ascoli as the next proposition shows (for a concrete example of a separable space $X$ satisfying the conditions of this proposition  see \cite[2.13]{Vaughan}).

\begin{proposition} \label{p:countably-compa-not-Ascoli}
Let $X$ be an infinite countably compact zero-dimensional $T_1$-space. If every compact subset of $X$ is finite, then $X$ is not sequentially $2$-Ascoli.
\end{proposition}

\begin{proof}
Since $X$ is infinite and Tychonoff, there is a sequence $\{U_{n}\}_{n\in\w}$ of clopen nonempty subsets of $X$ such that $U_n\cap U_m=\emptyset$ for all distinct $n,m\in\w$. For every $n\in\w$ take a point $x_n\in U_n$. Since $X$ is countably compact, there is a cluster point $z$ of the sequence $\{x_{n}\}_{n\in\w}$. Observe that $z\not\in \bigcup_{n\in\w} U_n$ because all $U_n$s are open and pairwise disjoint.

For every $n\in\w$, set $f_n:=\mathbf{1}_{U_n}$. Since all compact subsets of $X$ are finite, the pairwise disjointness of $U_n$s implies that $f_n\to \mathbf{0}$ in $\CC(X,\mathbf{2})$. On the other hand, since any neighborhood $V$ of $z$ contains some $x_{k}$ we have
$
f_k(x_k)-f_k(z)=1.
$
Therefore the null-sequence $\{f_{n}\}_{n\in\w}$ is not equicontinuous. Thus $X$ is not sequentially $2$-Ascoli.\qed
\end{proof}

Proposition \ref{p:P-space-2-Ascoli} and Theorem \ref{t:seq-compact-seq-Ascoli} motivate the following problem.
\begin{problem}
Does there exist a totally countably compact {\rm ((}near{\rm )} sequentially compact or $\w$-bounded{\rm )}  Tychonoff space $X$ which is not Ascoli?
\end{problem}



\section{A characterization of locally compact $\mu$-spaces}  \label{sec:product-Ascoli}


Essentially using a result of Whitehead \cite[Lemma~4]{Whitehead}, Cohen proved that, if $Z$ is a locally compact Hausdorff space, then $Z\times X$ is a $k$-space for every $k$-space $X$. In \cite{Michael-lc} Michael proved the converse assertion. So we obtain a characterization of locally compact spaces, see \cite[3.12.14(c)]{Eng}: {\em A Tychonoff space $Z$ is locally compact if and only if the product $Z\times X$ is a $k$-space for every $k$-space $X$.} This result motivates the following problem:
\begin{problem} \label{prob:product-Ascoli}
Is it true that  a Tychonoff space $Z$ is locally compact if and only if the product $Z\times X$ is a (sequentially) Ascoli space for every (sequentially) Ascoli space  $X$?
\end{problem}
In Theorem \ref{t:Ascoli-product-mu} below, for the case of Ascoli spaces,  we obtain a positive answer to this problem under additional assumption that $X$ is a $\mu$-space. Recall that a topological space is called a {\em $\mu$-space} if every  functionally bounded subset of $X$ is relatively compact (a  subset $A$ of $X$ is called {\em functionally bounded} if $f(A)$ is a bounded subset of $\IR$ for every $f\in C(X)$). First we prove the following analogue of the Whitehead--Cohen result mentioned above.

\begin{theorem} \label{t:product-lc-Ascoli}
Let $Z$ be a locally compact space. If $X$ is an Ascoli space, then so is $Z\times X$.
\end{theorem}

\begin{proof}
Let $K$ be a compact subset of $\CC(Z\times X)$. We have to show that $K$ is equicontinuous. Fix $\delta>0$ and a point $(z_0,x_0)\in Z\times X$.

Choose an open neighborhood $U$ of $z_0$ with compact closure.
Define a map $T: \overline{U}\times K \to \CC(X)$ by $T(z,f)(x):=f(z,x)\in \CC(X)$. We claim that $T$ is continuous. Indeed, fix an arbitrary point $(y,g)\in \overline{U}\times K$ and an open standard neighborhood $T(y,g)+ [C; \e]$ of the function $T(y,g)\in \CC(X)$, where $C\subseteq X$ is compact and $\e>0$. For every $x\in C$, choose an open neighborhood $V_x\subseteq \overline{U}$ of $y$ and an open neighborhood $O_x \subseteq X$ of $x$ such that
\begin{equation} \label{equ:product-Ascoli-1}
|g(z,t)-g(y,x)|<\tfrac{\e}{3}\; \mbox{ for all } \; z\in V_x \; \mbox{ and } \; t\in O_x.
\end{equation}
Since $C$ is compact, there are $x_1,\dots,x_s\in C$ such that $C\subseteq \bigcup_{i=1}^s O_{x_i}$. For every $x\in C$, choose $i(x)\in\{1,\dots,s\}$ such that $x\in O_{x_{i(x)}}$ and set $V:=\bigcap_{i=1}^s V_{x_i}\subseteq \overline{U}$. Then $V$ is an open neighborhood of $y$ in $\overline{U}$. Now, for every $z\in V$, $x\in C$ and each function $g+h\in \big(g+\big[ \overline{U}\times C; \tfrac{\e}{3}\big]\big)\cap K$, (\ref{equ:product-Ascoli-1}) implies
\[
\begin{aligned}
|T(z,g& +h)(x)  -T(y,g)(x)|  =|g(z,x)+h(z,x)-g(y,x)|\\
& \leq|g(z,x)-g(y,x_{i(x)})| + |g(y,x_{i(x)})-g(y,x)|+  |h(z,x)| <\tfrac{\e}{3}+\tfrac{\e}{3}+\tfrac{\e}{3}=\e.
\end{aligned}
\]
Thus $T$ is continuous, and hence the set $T(\overline{U}\times K)\subseteq \CC(X)$ is compact.

Since $X$ is an Ascoli space, it follows that the compact set $T(\overline{U}\times K)$ is equicontinuous. In particular,  there is an open neighborhood $\mathcal{O}$ of $x_0$ such that
\begin{equation} \label{equ:product-Ascoli-2}
|f(z,x)-f(z,x_0)|=|T(z,f)(x)-T(z,f)(x_0)| <\tfrac{\delta}{2} \; \mbox{ for all } \;  x\in \mathcal{O}, z\in  \overline{U} \; \mbox{ and } \; f\in K.
\end{equation}

Consider the restriction map $R: \CC(Z\times X)\to \CC(Z)$ defined by $R(f):= f(z,x_0)$. Then $R(K)$ is a compact subset of $\CC(Z)$. The space $Z$ being locally compact is Ascoli. Therefore there is an open neighborhood $W\subseteq U$ of $z_0$ such that
\begin{equation} \label{equ:product-Ascoli-3}
|f(z,x_0)-f(z_0,x_0)|=|R(f)(z)-R(f)(z_0)| <\tfrac{\delta}{2} \; \mbox{ for all } \; z\in  W \; \mbox{ and } \; f\in K.
\end{equation}
Now, for every $(z,x)\in W\times \mathcal{O}$ and each $f\in K$,  (\ref{equ:product-Ascoli-2}) and (\ref{equ:product-Ascoli-3}) imply
\[
|f(z,x)-f(z_0,x_0)|\leq |f(z,x)-f(z,x_0)|+|f(z,x_0)-f(z_0,x_0)|<\delta.
\]
Therefore $K$ is equicontinuous. Thus $Z\times X$ is an Ascoli space.\qed
\end{proof}




We say that a topological space $X$ is {\em locally functionally bounded} if every point of $X$ has a functionally bounded neighborhood.
To prove Theorem \ref{t:Ascoli-product-mu}  we need the following assertion.

\begin{proposition} \label{p:product-Ascoli}
If a  Tychonoff space $Z$ is not locally functionally bounded, then there is a Fr\'{e}chet--Urysohn fan $X$ such that the product $Z\times X$ is not an Ascoli space.
\end{proposition}

\begin{proof}
Let $z_0$ be a point of $Z$ which does not have a functionally bounded neighborhood. Fix an open base $\{ V_\alpha\}_{\alpha\in\AAA}$ of $z_0$. For every $\alpha\in\AAA$, choose a function $f_\alpha\in C(X)$ such that $f_\alpha{\restriction}_{V_\alpha}$ is unbounded. Then we can choose a discrete sequence  $\{ V_{n,\alpha}\}_{n\in\w}$ in $X$ such that all $V_{n,\alpha}$ are open subsets of $V_{\alpha}\SM \{z_0\}$. For every $n\in\w$, choose an arbitrary point $z_{n,\alpha}\in V_{n,\alpha}$. Consider the Fr\'{e}chet--Urysohn fan $X:=V(\AAA)$ with the unique non-isolated point $x_\infty$ (so $S_\alpha=\{ x_{n,\alpha}\}\cup\{ x_\infty\}$ and $x_{n,\alpha}\to x_\infty$ as $n\to\infty$). For each $\alpha\in\AAA$ and every $n\in\w$, set
\[
a_{n,\alpha}:= (z_{n,\alpha}, x_{n,\alpha})\in Z\times V(\AAA), \quad U_{n,\alpha}:= V_{n,\alpha} \times \{ x_{n,\alpha}\}.
\]
It is clear that $U_{n,\alpha}$ is an open neighborhood of $a_{n,\alpha}$ and $(z_0,x_\infty)$ is a cluster point of the family $\{a_{n,\alpha}:\alpha\in\AAA,n\in\w\}$. Therefore, to apply Proposition \ref{p:Ascoli-sufficient} we have to check that for every compact subset $C$ of $Z\times V(\AAA)$, the set
\[
\AAA_0 :=\{ (n,\alpha)\in \w\times \AAA: \; U_{n,\alpha}\cap C\not=\emptyset\}
\]
is finite. To this end, consider the projection $K:=\pi_{V(\AAA)}(C)$ of $C$ into $V(\AAA)$. Then $K$ is a compact subset of $V(\AAA)$, and hence there is a finite subset $I\subseteq \AAA$ such that  $\{ x_{n,\alpha}\}\cap K\not=\emptyset$ only if  $\alpha\in I$. Now,  for every $\alpha\in I$ the family $\{ \pi_Z(U_{n,\alpha})=V_{n,\alpha}: n\in\w\}$ of projections onto $Z$ is discrete in $Z$. Therefore,  for every $\alpha\in I$, the family
\[
\{ n\in\w: V_{n,\alpha} \cap \pi_Z(C)\not=\emptyset\}
\]
is finite. Thus $\AAA_0$ is finite.
Finally, Proposition \ref{p:Ascoli-sufficient} implies that the space $Z\times V(\AAA)$ is not Ascoli. \qed
\end{proof}

\begin{theorem} \label{t:Ascoli-product-mu}
For a $\mu$-space $Z$ the following assertions are equivalent:
\begin{enumerate}
\item[{\rm (i)}] $Z$ is locally compact;
\item[{\rm (ii)}] $Z\times X$ is an Ascoli space for each Ascoli space $X$;
\item[{\rm (iii)}] $Z\times X$ is an Ascoli space for each Fr\'{e}chet--Urysohn space $X$.
\end{enumerate}
\end{theorem}

\begin{proof}
The implication (i)$\Rightarrow$(ii) follows from Theorem \ref{t:product-lc-Ascoli}, and (ii)$\Rightarrow$(iii) is trivial.

(iii)$\Rightarrow$(i) By Proposition \ref{p:product-Ascoli}, the space $Z$ is locally functionally bounded. Fix a point $z\in Z$ and its functionally bounded neighborhood $V_z$. Since $Z$ is a $\mu$-space, it follows that the closure $\overline{V_z}$ of $V_z$ is a compact subset of $Z$. So $z$ has a compact neighborhood  $\overline{V_z}$. Thus $Z$ is locally compact.\qed
\end{proof}


\section{Sequentially Ascoli function spaces} \label{sec:seq-Ascoli-Cp} 


Let $X$ be a Tychonoff space and let $(Y,\rho)$ be a metric space. We denote by $Y^X$ the family of all maps from $X$ to $Y$,  and let $E$ be a subspace of $Y^X$. In particular, if $E$ is  endowed with the pointwise topology, then for a function $f\in E$, a finite subset $A$ of $X$ and $\e>0$, we set
\[
W[f;A,\e] := \big\{ g\in E: \rho\big(f(x),g(x)\big)<\e \mbox{ for every } x\in A\big\},
\]
so $W[f;A,\e]$ is a standard open neighborhood of $f$ in $E$.

If additionally $Y$ is a (metrizable) locally convex space, there is a long tradition to investigate locally convex properties of the space $C(X,Y)$ endowed with the pointwise topology or the compact-open topology by means of topological properties of $X$ and locally convex properties of $Y$. Of course the most important case is the case when $Y$ is a Banach space. For numerous results obtained in the eighties of the last century and historical remarks we refer the reader to the well known lecture notes of Schmets \cite{Schmets}.

The study of topological properties of spaces $C_p(X,Y)$ and  $\CC(X,Y)$, when $X=\IR, [0,1]$ or the doubleton $\mathbf{2}=\{0,1\}$, is one of the main direction in General Topology. We refer the reader to the books \cite{Arhangel,mcoy,Tkachuk-Book-2010} and the papers \cite{BG,Gabr-B1,Gabr-C2,GGKZ,GKP,GO,Pol-1974,Sak2,Sakai-3} and references therein.

Although the spaces $C_p(X,Y)$ and  $\CC(X,Y)$ are the most important, there are other important classes of (discontinuous) functions from $X$ to $Y$ widely studied in General Topology and Analysis as, for example, the classes of bounded/precompact continuous functions and the classes of Baire functions (we recall their definitions below). So the question of the study of topological properties of these classes of functions arises naturally. It turns out that the study of the aforementioned classes of functions can be realized simultaneously using the following idea successfully applied in  \cite{Gabr-B1}: instead the whole space $C(X,Y)$ we consider only some of its sufficiently rich subspaces in the sense of Definitions \ref{def:relatively-Y-Tych} and \ref{def:rel-p-Tych}.

The main purpose of this section is to characterize those spaces $X$ for which the function space $C_p(X,Y)$ and classes of Baire functions are (sequentially) Ascoli, also we obtain a necessary condition on $X$ for which the space $\CC(X)$ is sequentially Ascoli. To this end, we shall essentially use the ideas from \cite{Gabr-B1,GGKZ} and the papers of Sakai \cite{Sak2,Sakai-3}; in fact (the proofs of) Lemmas \ref{l:Cp-Ascoli-1} and \ref{l:Cp-Ascoli-2},  Proposition \ref{p:Cp-seq-Ascoli} and  Proposition \ref{p:Cp-kappa-FU} given below are corresponding modifications and extensions (of the proofs) of Lemmas 2.1 and 2.2 of \cite{GGKZ}, Theorem~1.1 of \cite{GGKZ} and  Theorem~2.1 of \cite{Sak2}, and of Theorem 2.1 of \cite{Sak2}, respectively (nevertheless we give their detailed proofs for the reader convenience and to make the paper self-contained).

\begin{lemma} \label{l:Cp-Ascoli-1}
Let  $(Y,\rho)$ be a metric space, $X$ be a set, $E$ be a dense subspace of the topological product $Y^X$, and let $g\in E$. For every $n\in\w$, let $U_n$ be an open subset of $E$ such that $g\in\overline{U_n}$, and let $\W_n$ be an open cover of $U_n$. Then for every $n\in\w$, there exists $W_n\in\W_n$ such that $g\in\overline{\bigcup\{W_n:n\in\w\}}$.
\end{lemma}

\begin{proof}
We proceed by induction on $n\in\w$. For $n=0$, choose arbitrarily $W_0\in \W_0$ and $h_0\in W_0$. Since $W_0$ is open, we can choose a finite subset $A_0$ of $X$ and $\e_0>0$ such that $W[h_0;A_0,\e_0]\subseteq W_0$. For $n=1$, choose $W_1\in\W_1$ such that $W_1 \cap W[g;A_0, 1]$ is not empty (this is possible since $g\in\overline{U_1}$ and $\W_1$ covers $U_1$). Take arbitrarily an $h_1\in W_1 \cap W[g;A_0, 1]$, a finite subset $A_1$ of $X$ containing $A_0$ and $\e_1$, $0<\e_1<\tfrac{\e_0}{2}$, such that
\[
W[h_1;A_1,\e_1]\subseteq W_1 \cap W[g;A_0, 1].
\]
Continuing this process we can construct an increasing sequence $\{ A_n:n\in\w \}$ of finite subsets of $X$, a decreasing null-sequence $\{\e_n:n\in\w\}$ of
positive reals, a sequence $\{ W_n \in\W_n: n\in\w\}$ of open subsets of $E$  and a sequence $\{ h_n:n\in\w\}$ in $E$ such that
\begin{equation} \label{equ:seq-Ascoli-A}
W\big[h_{n+1};A_{n+1},\e_{n+1}\big] \subseteq W_{n+1} \cap W\big[g;A_n, \tfrac{1}{n+1} \big] \quad \mbox{ for all } \; n\in\w.
\end{equation}
We claim that $\{ W_n:n\in\w\}$ is as required. Indeed, fix a finite $F\subseteq X$ and $\e >0$. Choose a natural number $n_0 >1$ such that
$F\cap(\bigcup_{n\in\w}A_n)\subseteq A_{n_0}$ and $\frac{1}{n_0}+\e_{n_0+1}<\e$. Since $E$ is dense in $Y^X$, there is an $h\in E$ such
that
\[
\rho\big(h(x),h_{n_0+1}(x)\big)<\e_{n_0+1} \mbox{ for every } x\in  A_{n_0+1},
\]
(so, by (\ref{equ:seq-Ascoli-A}), $\rho\big(h(x),g(x)\big)<\tfrac{1}{n_0+1}+\e_{n_0+1}<\e$ for every $x\in F\cap A_{n_0}=F\cap A_{n_0+1}$) and
\[
\rho\big(h(x),g(x)\big)<\e \mbox{ for every } x\in F\setminus A_{n_0+1}.
\]
Then $h\in W\big[ h_{n_0+1}; A_{n_0+1},\e_{n_0+1}\big] \subseteq W_{n_0+1}$ and $h\in W[g;F,\e]$. Thus $h\in  W_{n_0+1}\cap W[g;F,\e]$. \qed
\end{proof}

\begin{lemma} \label{l:Cp-Ascoli-2}
Let  $(Y,\rho)$ be a metric space containing at least two points, $X$ be a set, $E$ be a dense subspace of the topological product $Y^X$, and let $g\in E$. Assume that $E$ is a sequentially Ascoli space and $\{U_n:n\in\w\}$ is a sequence of open subsets of $E$ such  that $g\in\overline{\bigcup\{U_n:n\in\w\}}$ but $g\not\in\overline{U_n}$ for all $n\in\w$. Then there exists a compact subspace $K$ of $E$ such that the set $\{n\in\w :K\cap U_n\neq\emptyset\}$ is infinite.
\end{lemma}

\begin{proof}
Suppose for a contradiction that for every compact subset $K$ of $E$, $K\cap U_n\neq\emptyset$ only for finitely many $n\in\w$. Since $E$ is Tychonoff,
for every $n\in\w$ and each $h\in U_n$, we can choose a  function $\phi_{h,n}\in C\big( E\big)$ such that
\[
\phi_{h,n}(h) >1 \; \mbox{ and } \; \phi_{h,n}{\restriction}_{E\setminus U_n}=0,
\]
and set $W_{h,n}:= \{ f\in E: \phi_{h,n}(f)>1\}$, it is clear that $W_{h,n}$ is an open subset of $U_n$. For every $n\in\w$, set
$\widetilde{U}_n := \bigcup_{l\geq n}U_l$ and
\[
\W_n :=\left\{ W_{h,l}: l \geq n \; \mbox{ and } \; h\in U_l \right\}.
\]
Then $g\in \overline{\widetilde{U}_n}$ and $\W_n$ is an open cover of $\widetilde{U}_n$. Applying Lemma~\ref{l:Cp-Ascoli-1} for the sequence
$\{ \widetilde{U}_n: n\in\w\}$ we obtain the following: for every $n\in\w$ there exists $W_n\in\W_n$ such that $g\in\overline{\bigcup\{W_n:n\in\w\}}$.
Let $l_n\geq n$ and $\phi_n$ be witnesses  for $W_n\in\W_n$ (observe also that $g\not\in\overline{W_n}$ for all $n\in\w$).

We claim that $\phi_n$ converges to the zero function $\mathbf{0}\in \CC(E)$. Indeed, fix a compact $K\subseteq E$ and $\e>0$. By our assumption the set
$F:=\{ n\in\w: K\cap U_{l_n}\not=\emptyset\}$ is finite. Therefore, for every $n\in\w\setminus F$, we have $\phi_n{\restriction}_K =0$ and hence
$\phi_n\in W[\mathbf{0};K,\e]$. Thus $\phi_n\to \mathbf{0}$ in $\CC(E)$.

On the other hand, given any open $V\subseteq E$ containing $g$ and $m\in\w$, the facts $g\in\overline{\bigcup\{W_n:n\in\w\}}$ and $g\not\in\overline{W_n}$ imply
that there exist $n\geq m$ and $f\in V\cap W_n$ such that $\phi_n(f)>1$. This proves that the convergent sequence
\[
\{\phi_n:n\in\w\}\cup\{\mathbf{0}\}\subseteq \CC(E)
\]
is not equicontinuous at $g$. Thus $E$ is not sequentially Ascoli, a contradiction.\qed
\end{proof}

\begin{definition}[\cite{GO}] \label{def:rel-p-Tych} {\em
Let $X$ and $Y$ be topological spaces. A subspace $E$  of $C_p(X,Y)$ is called {\em relatively $Y_p$-Tychonoff} if for every closed subset $A$ of $X$, finite subset $F\subseteq X$ contained in $X\SM A$, point $y_0\in Y$ and each $f\in C_p(X,Y)$  there is a function ${\bar f}\in E$ such that
$
{\bar f}{\restriction}_F=f{\restriction}_F  \mbox{ and }  {\bar f}(A)\subseteq \{ y_0\}.
$\qed }
\end{definition}
In other words, $E$ is a relatively $Y_p$-Tychonoff   subspace of $C_p(X,Y)$ if for every closed subset $A$ of $X$, every function $f\in C_p(X,Y)$ and each finite $F\subseteq X$ such that  $F\subseteq X\SM A$, the restriction $f{\restriction}_F$ can be extended to a function ${\bar f}\in E$ with an additional condition ${\bar f}(A)\subseteq \{ y_0\}$ for some point $y_0\in Y$.

For a subset $A$ of a topological space $X$, let $[A]_0:=A$ and for each $\alpha\in\w_1$, let $[A]_\alpha$ be the set of limits of convergent sequences in $\bigcup_{i\in\alpha} [A]_i$. The set $[A]_s:=\bigcup_{\alpha\in\w_1} [A]_\alpha$ is called the {\em sequential closure } of $A$.

A family  $\{ A_i\}_{i\in I}$ of subsets of a set $X$ is called {\em point-finite} if for every $x\in X$, the set $\{ i\in I: x\in A_i\}$ is finite. A family  $\{ A_i\}_{i\in I}$ of subsets of a topological space $X$ is said to be {\em strongly point-finite} if for every $i\in I$ there is an open set $U_i$ of $X$  such that $A_i\subseteq U_i$ and  the family $\{ U_i: i\in I\}$ is point-finite.

Following Sakai \cite{Sak2}, a topological space $X$ is said to have the {\em property $(\kappa)$} if every pairwise disjoint sequence of finite subsets of $X$ has a strongly point-finite subsequence.

\begin{proposition} \label{p:Cp-seq-Ascoli}
Let $Y$ be a metrizable non-compact space, $X$ be a $Y$-Tychonoff space, and let $E$ be a relatively $Y_p$-Tychonoff   subspace of $C_p(X,Y)$ containing all constant functions. Consider the following statements:
\begin{enumerate}
\item[{\rm (i)}] $E$ is a sequentially Ascoli space;
\item[{\rm (ii)}] the sequential closure of any open set of $E$ is closed;
\item[{\rm (iii)}] $X$ has the property $(\kappa)$.
\end{enumerate}
Then  {\rm (i)$\Rightarrow$(iii)}  and {\rm (ii)$\Rightarrow$(iii).}
\end{proposition}

\begin{proof}
Since $Y$ is not compact, one can find a sequence $\{ g_n\}_{n\in\w}\subseteq Y$ and a sequence $\{ V_n\}_{n\in\w}$ of open subsets in $Y$ such that
\begin{enumerate}
\item[{\rm (a)}] $g_n\in V_n$ for all $n\in\w$, and
\item[{\rm (b)}] the family $\V=\{ \overline{V_n}\}_{n\in\w}$ is discrete in $Y$.
\end{enumerate}

(i)$\Rightarrow$(iii) Assume that $E$ is a sequentially Ascoli space. Consider a sequence $\AAA=\{ A_n:n\in\w\}$ of finite subsets of $X$ such that $A_n\cap A_m=\emptyset$ for all $n\neq m$. We have to find  an infinite $J\subseteq \w$ and open sets $U_j\supseteq A_j$ for all $j\in J$, such that the family $\{U_j:j\in J\}$ is point-finite.

For every $k\in\w$, set $O_k :=E\cap \bigcap_{a\in A_k} W[\mathbf{g}_k;\{a\},V_k]$. It follows from (a) and (b) that $\mathbf{g}_0\not\in \overline{O_k}$ for all $k>0$. Since $X$ is $Y$-Tychonoff and $E$ is dense in $C_p(X,Y)$ (so $E$ is dense in $Y^X$), the disjointness of the sequence $\AAA$ implies $\mathbf{g}_0\in\overline{\bigcup\{O_k:k>0\}}^E$. By Lemma~\ref{l:Cp-Ascoli-2},  there exists a compact set $K\subseteq E$ intersecting infinitely many of the $O_k$'s, so the set $J:=\{ n>0: K\cap O_n\not= \emptyset\}$ is infinite. For every $j\in J$, choose a function $h_j\in K\cap O_j$ and set
\[
U_j :=\{x\in X: h_j(x)\in V_j \},
\]
and note that $U_j$ is an open subset of $X$ containing $A_j$. We show that $\{ U_j\}_{j\in J}$ is point-finite.
To this end, fix an arbitrary $x\in X$ and set $J'_x :=\{ j\in J: x\in U_j\}$. We have to prove that $J'_x$ is finite.

Suppose for a contradiction that $J'_x$ is infinite. Since $h_j(x)\in V_j $, (b) implies that the sequence $\{h_j(x):j\in J'_x\}$ is closed and discrete in $Y$. On the other hand, since $\{h_j:j\in J'_x\}$ is a subset of the compact set $K$, the sequence $\{h_j(x):j\in J'_x\} \subseteq Y$ must be relatively compact. This contradiction shows that $J'_x$ must be finite as desired.
\smallskip

(ii)$\Rightarrow$(iii) Assume that the sequential closure of any open set of $E$ is closed. Consider a disjoint sequence $\{ A_n:n\in\w\}$ of finite subsets of $X$. We have to find  an infinite $J\subseteq \w$ and open sets $U_j\supseteq A_j$ for all $j\in J$, such that $\{U_j:j\in J\}$ is point-finite.
Suppose for a contradiction that no subsequence of $\{ A_n:n\in\w\}$ is strongly point-finite.

For each $n\in\w$, set
\[
U_n :=\{ f\in E: f(A_n) \subseteq V_n \} =E\cap \bigcap_{a\in A_n} [\{a\};V_n],
\]
and let $U:=\bigcup_{n>0} U_n$. Since each $U_n$ is open in $E$, the set $U$ is also open in $E$. Observe that
\[
\overline{U_n}\subseteq \{ f\in E: f(A_n) \subseteq \overline{V_{n}}\}.
\]

{\em Claim 1. $[U]_s =\bigcup_{n>0} [U_n]_s \subseteq \bigcup_{n>0} \overline{U_n}$.} Indeed, let $f_j\in U_{n_j}$ for $0<n_0<n_1<\cdots$. Then $A_{n_j} \subseteq f^{-1}_j(V_{n_j})$, and hence, by our supposition that no subsequence of $\{ A_n:n\in\w\}$ is strongly point-finite, the family $\big\{ f^{-1}_j(V_{n_j})\big\}_{j\in\w}$ is not point-finite. Therefore there exists a point $z\in X$ which is contained in $f^{-1}_j(V_{n_j})$ for infinitely many $j$'s. This means that $f_j(z)\in V_{n_j} $ for infinitely many $j$'s. But then, taking into account that the family $\mathcal{V}' :=\{ \overline{V_{n_j}}\}$ is discrete and hence $\overline{\bigcup\mathcal{V}'}=\bigcup\overline{V_{n_j}}$, we obtain that $f_j(z)$ does not converge in $Y$. Thus, if a sequence $S=\{ f_j\}_{j\in\w}\subseteq U$ converges in $E$, there is an $m\in\w$ such that $S\cap U_{m}$ is infinite and hence $S$ converges to a point of $[U_m]_s$. The claim is proved.

By (a) and (b) we have $g_0\not\in \overline{V_{n}}$ for every $n>0$. It follows that $\mathbf{g}_0\not\in \bigcup_{n>0} \overline{U_n}$. Hence, by Claim 1, $\mathbf{g}_0\not\in [U]_s$. Therefore, to get a desired contradiction it suffices to prove that $\mathbf{g}_0\in \overline{U}$ (which means that the sequential closure of $U$ is not closed).

Fix a finite subset $F$ of $X$ and $\delta >0$. Since $A_n\cap A_m=\emptyset$ for all $n\neq m$, there is $k\in\w$ such that $F\cap A_k=\emptyset$. Since $X$ is $Y$-Tychonoff and $E$ is a relatively $Y_p$-Tychonoff   subspace of $C_p(X,Y)$, one can find a function $f\in E$ such that $f{\restriction}_F=\mathbf{g}_0{\restriction}_F$ and $f(A_k)=\{g_k\}\subseteq V_k $. Thus $f\in W[\mathbf{g}_0;F,\delta]\cap U_k$ and hence $\mathbf{g}_0\in \overline{U}$.\qed
\end{proof}

A class of spaces with the property $(\kappa)$ is given below.
\begin{proposition} \label{p:kappa-P-space}
Every Tychonoff $P$-space $X$ has the property $(\kappa)$.
\end{proposition}

\begin{proof}
First we recall that every Tychonoff $P$-space is zero-dimensional. Observe that every countable subset of a Tychonoff $P$-space being an $F_\sigma$-set is closed and discrete. Now let $\{ A_n:n\in\w\}$ be a disjoint sequence of finite subsets of $X$.
Set $A:=\bigcup_{n\in\w} A_n$. For every $a\in A$ choose a clopen neighborhood $V_a$ of $a$ such that the sequence $\{ V_a:a\in A\}$ is disjoint. For every $n\in\w$, set $U_n:=\bigcup_{a\in A_n} V_a$. It is clear that the sequence $\{U_n:n\in\w\}$ is point-finite. Thus $X$ has the property $(\kappa)$.\qed
\end{proof}

Now we show that  in the case when $Y$ is a discrete metrizable group the property $(\kappa)$ is also sufficient in Proposition \ref{p:Cp-seq-Ascoli}.

\begin{proposition} \label{p:Cp-kappa-FU-discrete}
Let $Y$ be a discrete abelian group containing more than one point, $X$ be an infinite $Y$-Tychonoff space, and let $E$ be a relatively $Y_p$-Tychonoff subgroup of $C_p(X,Y)$  containing all constant functions. If $X$ has the property $(\kappa)$, then  $E$ is a $\kappa$-Fr\'{e}chet--Urysohn space.
\end{proposition}

\begin{proof}
Let $U$ be an open subset of $E$ and let $g\in \overline{U}\SM U$. Since the group $E$ is homogeneous, without loss of generality we can assume that $g$ is the zero-function $\mathbf{0}$. So, let $\mathbf{0}\in \overline{U}\setminus U$. We have to find a sequence in $U$ converging to $\mathbf{0}$.
Since $Y$ is discrete, for every $f\in U$, there is  a finite subset $F(f)$ of $X$ such that $\bigcap_{x\in F(f)} W[f;\{x\},0] \subseteq U$.  Set
\[
A(f):=\{ x\in F(f): f(x)=0\} \; \mbox{ and }\; B(f):=\{ x\in F(f): f(x)\not= 0\},
\]
and note that $B(f)\not=\emptyset$ for every $f\in U$ because $\mathbf{0} \not\in U$.

Now we construct a sequence $\{ f_n\}_{n\in\w} \subseteq U$ inductively, as follows. Take an arbitrary $f_0\in U$.
For $n=1$, the inclusion $\mathbf{0}\in \overline{U}$ implies that there is $f_1\in U\cap \bigcap_{x\in F(f_0)} W[f;\{x\},0]$.  Continuing this process, for every $n>0$ we can find  a function $f_n \in U\cap \bigcap \big\{ W\big[\mathbf{0};\{x\},0\big]: x\in \bigcup_{i=0}^{n-1} F(f_i)\big\}$.
By the choice of $f_n$, for every $n>0$ we have
\[
B(f_n)\cap \big( B(f_0)\cup\cdots\cup B(f_{n-1})\big)\subseteq B(f_n)\cap \bigcup_{i=0}^{n-1} F(f_i) =\emptyset.
\]
So the sequence $\{ B(f_n)\}_{n>0}$ is pairwise disjoint and hence, by the property $(\kappa)$, we can take a strongly point-finite subsequence $\{ B(f_{n_j}) \}_{j\in\w}$. Let $\{ O_j\}_{j\in\w}$ be a point-finite open family of $X$ such that
\begin{equation} \label{equ:kFU-Ascoli-1}
B(f_{n_j}) \subseteq O_j \; \mbox{ and } \; \overline{O_j} \cap  A(f_{n_j}) =\emptyset \quad (j\in\w).
\end{equation}
Since $X$ is $Y$-Tychonoff and $E$ is a relatively $Y_p$-Tychonoff   subspace of $C_p(X,Y)$, for every $j\in\w$ there is a function $t_j\in E$ such that
\begin{equation} \label{equ:kFU-Ascoli-2}
t_j\big( X\setminus O_j\big) =\{ 0\} \; \mbox{ and } \; t_j(z)= f_{n_j}(z) \; \mbox{ for all } z\in B(f_{n_j}) .
\end{equation}
Observe that if $x\in A(f_{n_j})$, (\ref{equ:kFU-Ascoli-1}) implies that $t_j(x)=0=f_{n_j}(x)$. Therefore $t_j{\restriction}_{F(f_{n_j})} =f_{n_j}{\restriction}_{F(f_{n_j})}$ for every $j\in\w$. Thus $t_j\in \bigcap_{x\in F(f_{n_j})} W[f_{n_j};\{x\},0] \subseteq U$ for all $j\in\w$. Finally, since $\{ O_j\}_{j\in\w}$ is point-finite, (\ref{equ:kFU-Ascoli-2}) implies that $t_j \to \mathbf{0}$.\qed
\end{proof}

\begin{corollary} \label{c:Cp-kappa-FU-discrete}
Let $Y$ be a discrete space containing more than one point, and let  $X$ be a zero-dimensional $T_1$-space. If $X$ has the property $(\kappa)$, then  $C_p(X,Y)$ is a $\kappa$-Fr\'{e}chet--Urysohn space.
\end{corollary}

\begin{proof}
By Proposition 2.7 of \cite{BG-Baire}, the space $X$ is $Y$-Tychonoff.
Since every discrete space is homeomorphic to a discrete abelian group of the same cardinality, the assertion follows from Proposition \ref{p:Cp-kappa-FU-discrete} applied to $E=C_p(X,Y)$.\qed
\end{proof}

The proof of Proposition \ref{p:Cp-kappa-FU-discrete} shows that only the following two moments are essential: (1) the sequence $\{ B(f_n)\}_{n>0}$ is disjoint, and (2) one can find a function $t_j$ such that $t_j{\restriction}_{F(f_{n_j})} =f_{n_j}{\restriction}_{F(f_{n_j})}$. To get useful variations of these two properties in the case of non-discrete group  $Y$ we need the following notion which is stronger than the property of being a $Y_\I$-Tychonoff space.

\begin{definition} \label{def:Y-I-approx} {\em
Let $X$ and $Y$ be  topological spaces and let $\I$ be an ideal of compact sets of $X$. A subspace $H$ of $Y^X$ is said to have an {\em $\I$-approximation property at a point $y_0$} if the point $y_0$ has an open base $\V$ of neighborhoods such that
for every $V\in\V$, closed subset $A$ of $X$, $F\in \I$ contained in $X\SM A$, and a function $f:F\to V$ with finite image there is a function ${\bar f}\in H$ such that
\[
{\bar f}(X)\subseteq V, \quad {\bar f}{\restriction}_F=f \quad \mbox{ and } \quad {\bar f}(A)\subseteq \{ y_0\}.
\]
The space $X$ has the {\em $\I$-approximation property} if it has the $\I$-approximation property at each point $y\in Y$.\qed }
\end{definition}
As usual if $\I=\FF(X), \mathcal{S}(X)$ or $\KK(X)$, we shall say that the space $X$ has the $p$-, $s$- or {\em $k$-approximation property}, respectively. If $Y=\IR$ (or any locally arc-connected space), it follows from Proposition 2.9 of \cite{GO} that a topological space $X$ has the $\I$-approximation property if and only if $X$ is Tychonoff.

\begin{example} \label{exa:zero-dim-Yp-approx} {\em
Let $X$ be a zero-dimensional $T_1$-space, and let $\I$ be an ideal of compact sets of $X$. Then $X$ has the $\I$-approximation property for every nonempty Tychonoff space $Z$ since, by Proposition 
2.10 of \cite{GO},   $X$ is $Y_\I$-Tychonoff for every $T_1$-space $Y$.}\qed
\end{example}

Let $X$ be a Tychonoff space and let $E$ be a locally convex space. We denote by $C^b(X,E)$ and $C^{rc}(X,E)$ the spaces of all functions $f\in C(X,E)$ such that $f(X)$ is a bounded or a relatively compact subset of $E$, respectively.

\begin{proposition} \label{p:Cb-Yp-approx}
Let $E$ be a non-trivial locally convex space, $X$ be a Tychonoff space, and let $\I$ be an  ideal of compact sets of $X$. Then:
\begin{enumerate}
\item[{\rm (i)}] $X$ is $E_\I$-Tychonoff;
\item[{\rm (ii)}]  $C^b_\I(X,E)$ and $C^{rc}_\I(X,E)$ are  relatively $E_\I$-Tychonoff subspaces of $C_\I(X,E)$;
\item[{\rm (iii)}]  $C_\I(X,E)$, $C^b_\I(X,E)$ and $C^{rc}_\I(X,E)$ have the $\I$-approximation property.
\end{enumerate}
\end{proposition}

\begin{proof}
(i) and (iii) follow from (ii) of Proposition 2.9 in \cite{GO} (recall that any locally convex space has a base at zero consisting of absolutely convex open neighborhoods), and (ii) is proved in Proposition 4.4 of \cite{GO}.\qed
\end{proof}

\begin{proposition} \label{p:Cp-kappa-FU}
Let $Y$ be a non-discrete metric abelian group, $X$ be an infinite $Y$-Tychonoff space, and let $E$ be a subgroup of $C_p(X,Y)$ satisfying the following properties:
\begin{enumerate}
\item[{\rm (i)}] $E$ is a relatively $Y_p$-Tychonoff   subspace of $C_p(X,Y)$ containing constant functions;
\item[{\rm (ii)}] $E$ has the $p$-approximation property.
\end{enumerate}
If $X$ has the property $(\kappa)$, then  $E$ is a $\kappa$-Fr\'{e}chet--Urysohn space.
\end{proposition}

\begin{proof}
Let $U$ be an open subset of $E$ and let $g\in \overline{U}\SM U$. Since the group $E$ is homogeneous, without loss of generality we can assume that $g$ is the zero-function $\mathbf{0}$. So, let $\mathbf{0}\in \overline{U}\setminus U$. We have to find a sequence in $U$ converging to $\mathbf{0}$.
Fix a decreasing basis $\{V_n\}_{n\in\w}$ of open neighborhoods of $0\in Y$ witnessing  the $p$-approximation property.

Now we construct a sequence $\{ f_n\}_{n\in\w} \subseteq U$ inductively as follows.  Take an arbitrary $f_0\in U$. Choose a finite subset $F(f_0)$ of $X$ and  open sets $\{U_x: x\in F(f_0)\}$ in $Y$  such that
\[
E\cap \bigcap_{x\in F(f_0)}W[f_0;\{x\},U_x] \subseteq U
\]
(in what follows we shall omit ``$E\cap$'' to simplify the notations). Define
\[
A(f_0):=\{ x\in F(f_0): f_0(x)=0\} \; \mbox{ and }\; B(f_0):=\{ x\in F(f_0): f_0(x)\not= 0\},
\]
and note that $B(f_0)\not=\emptyset$ because $\mathbf{0} \not\in U$.

For $n=1$, the inclusion $\mathbf{0}\in \overline{U}$ implies that there is an $h\in U\cap \bigcap_{x\in F(f_0)}W[\mathbf{0};\{x\},V_{1}]$.
Choose a finite subset $F$ of $X$ and open sets $\{U_x: x\in F\}$ in $Y$ such that $F(f_0)\subsetneqq F$ (this is possible since $X$ is infinite) and
\[
\bigcap_{x\in F} W[h;\{x\},U_x]\subseteq U\cap \bigcap_{x\in F(f_0)}W[\mathbf{0};\{x\},V_{1}].
\]
In particular,
\begin{equation} \label{equ:seq-Ascoli-00}
h(x)\in V_1 \;\; \mbox{ for every }\; x\in B(f_0).
\end{equation}

Set
\[
A(h):=\{ x\in F: h(x)=0\} \; \mbox{ and }\; B(h):=\{ x\in F: h(x)\not= 0\},
\]
and note that $B(h)\not=\emptyset$ because $\mathbf{0} \not\in U$.
If the set $B(h) \SM B(f_0)$ is not empty, we set $f_1:=h$ and $C_1:=B(f_1) \SM B(f_0)$. Assume now that $B(h) \SM B(f_0)=\emptyset$. Then $h(x)=0$ for every $x\in F\SM B(f_0)$, and hence the set $V_{1} \cap \bigcap_{x\in F\SM B(f_0)} U_x$ is an open neighborhood of $0\in Y$. Since $Y$ is non-discrete, take an arbitrary nonzero point $y\in V_{1} \cap \bigcap_{x\in F\SM B(f_0)} U_x$. As $X$ is $Y$-Tychonoff and $E$ is a relatively $Y_p$-Tychonoff   subspace of $C_p(X,Y)$, there is an $f_1 \in E$ such that (recall that $F(f_0)\subsetneqq F$)
\[
f_1(x)=y \; \mbox{ if } \; x\in F\SM B(f_0), \; \mbox{ and } f_1(x)=h(x) \; \mbox{ for } \; x\in F\cap B(f_0)=B(f_0).
\]
By construction we have $f_1\in \bigcap_{x\in F} W[h;\{x\},U_x]\subseteq U$. So there is a finite subset $F(f_1)$ of $X$ and open sets $\{U_{x,1}: x\in F(f_1)\}$ in $Y$ such that $F\subseteq F(f_1)$ and
\[
\bigcap_{x\in F(f_1)} W[f_1;\{x\},U_{x,1}]\subseteq U\cap \bigcap_{x\in F(f_0)}W[\mathbf{0};\{x\},V_{1}].
\]
Set
\[
A(f_1):=\{ x\in F(f_1): f_1(x)=0\}, \quad B(f_1):=\{ x\in F(f_1): f_1(x)\not= 0\}, 
\]
and $C_1:=B(f_1) \SM B(f_0)$. Note that $C_1\not=\emptyset$ because, by construction, $\emptyset \not= F\SM B(f_0)\subseteq C_1$. If $z\in B(f_1)\cap B(f_0)$, then the choice of $f_1$ implies that $f_1(z)=h(z)$ and hence, by (\ref{equ:seq-Ascoli-00}), $f_1(z)\in V_1$.

Continuing this process, for every $n>0$ we can find  a function
\[
f_n \in U\cap \bigcap \left\{ W\big[\mathbf{0};\{x\},V_n\big]: \; x\in \bigcup_{i=0}^{n-1} F(f_i)\right\},
\]
 a finite set $F(f_n)\subseteq X$, and open sets $\{U_{x,n}: x\in F(f_n)\}$ such that
\begin{enumerate}
\item[(a)] $\bigcup_{i=0}^{n-1} F(f_i) \subsetneqq F(f_n)$ and $\bigcap_{x\in F(f_n)} W[f_n;\{x\},U_{x,n}] \subseteq U$;
\item[(b)] the set $C_n:=B(f_n)\SM \big( B(f_0)\cup\cdots\cup B(f_{n-1})\big)$ is not empty;
\item[(c)] if $ z\in B(f_n)\cap \big( B(f_0)\cup\cdots\cup B(f_{n-1})\big)$, then  $f_n(z)\in V_{n}$.
\end{enumerate}
For every $n>0$, set $D_n := B(f_n)\setminus C_n= B(f_n)\cap \big( B(f_0)\cup\cdots\cup B(f_{n-1})\big).$

For the pairwise disjoint sequence $\{ C_n\}_{n>0}$, we can take a strongly point-finite subsequence $\{ C_{n_j}\}_{j\in\w}$. Let $\{ O_j\}_{j\in\w}$ be a point-finite open family of $X$ such that
\[
C_{n_j} \subseteq O_j \; \mbox{ and } \; \overline{O_j} \cap \big( D_{n_j}\cup A(f_{n_j})\big) =\emptyset \quad (j\in\w).
\]
For every $j\in\w$, choose an open subset $W_j$ of $X$ such that
\[
D_{n_j} \subseteq W_j \; \mbox{ and } \; \overline{W_j} \cap \big( O_{j}\cup A(f_{n_j})\big) =\emptyset.
\]

Since $X$ is $Y$-Tychonoff and $E$ is a relatively $Y_p$-Tychonoff   subspace of $C_p(X,Y)$, for every $j\in\w$ there is a function $t_j\in E$ such that
\begin{equation} \label{equ:seq-Ascoli-4}
t_j\big( X\setminus O_j\big) =\{ 0\} \; \mbox{ and } \; t_j(z)= f_{n_j}(z) \; \mbox{ for all } z\in C_{n_j}.
\end{equation}
Since $\{ O_n\}_{n\in\w}$ is point-finite, it follows that $t_j \to \mathbf{0}$ in $C_p(X,Y)$.

Now the choice of the sequence $\{ V_n\}_{n\in\w}$, (ii) and (c) imply that for every $j\in\w$, there is an $h_j\in E$ such that
\begin{equation} \label{equ:seq-Ascoli-5}
h_j(X)\subseteq V_{n_j}, \; h_j\big( X\setminus W_j\big) =\{ 0\} \; \mbox{ and } \; h_j(z)= f_{n_j}(z) \; \mbox{ for all } z\in D_{n_j}.
\end{equation}
In particular, the sequence $\{ h_j\}_{j\in\w}$ uniformly converges to $\mathbf{0}$. Since $E$ is a group, we obtain that $\{ t_j+h_j\}_{j\in\w}\subseteq E$ and $t_j+h_j\to \mathbf{0}$ in $E$. The choice of $O$ and $W_j$ and  (\ref{equ:seq-Ascoli-4}) and (\ref{equ:seq-Ascoli-5}) imply $(t_j+h_j){\restriction}_{F(f_{n_j})} = f_{n_j}{\restriction}_{F(f_{n_j})}$, and hence
\[
t_j+h_j \in \bigcap_{ x\in F(f_{n_j})} W[f_{n_j};\{x\},U_{x,n_j}] \stackrel{(a)}{\subseteq} U \quad \mbox{ for all }\; j\in\w.
\]
 Thus $E$ is $\kappa$-Fr\'{e}chet--Urysohn.\qed
\end{proof}

Now we prove the main result of this section.

\begin{theorem} \label{t:Cp-seq-Ascoli}
Let $Y$ be a metrizable non-compact abelian group, $X$ be a $Y$-Tychonoff space, and let $E$ be a subgroup of $C_p(X,Y)$ satisfying the following properties:
\begin{enumerate}
\item[{\rm (a)}] $E$ is a relatively $Y_p$-Tychonoff   subspace of $C_p(X,Y)$ containing all constant functions;
\item[{\rm (b)}] if $Y$ is non-discrete, then $E$ has the $p$-approximation property.
\end{enumerate}.
Then the following assertions are equivalent:
\begin{enumerate}
\item[{\rm (i)}] $E$ is a $\kappa$-Fr\'{e}chet--Urysohn space;
\item[{\rm (ii)}] $E$ is an Ascoli space;
\item[{\rm (iii)}] $E$ is a sequentially Ascoli space;
\item[{\rm (iv)}] the sequential closure of any open set of $E$ is closed;
\item[{\rm (v)}] $X$ has the property $(\kappa)$.
\end{enumerate}
\end{theorem}

\begin{proof}
The implication (i)$\Rightarrow$(ii) follows from Theorem 2.5 of \cite{Gabr-B1} which states that every $\kappa$-Fr\'{e}chet--Urysohn space is Ascoli. 
The implications (i)$\Rightarrow$(iv) and (ii)$\Rightarrow$(iii) are trivial. The implications (iii)$\Rightarrow$(v) and (iv)$\Rightarrow$(v) follow from Proposition \ref{p:Cp-seq-Ascoli}. Finally, (v) implies (i) by Proposition \ref{p:Cp-kappa-FU}.\qed
\end{proof}

\begin{theorem} \label{t:Cb-Ascoli}
Let $Y$ be a non-trivial metrizable locally convex space, $X$ be a Tychonoff space, and let $E=C_p(X,Y)$, $C^b_p(X,Y)$ or $E=C^{rc}_p(X,Y)$. Then the following assertions are equivalent:
\begin{enumerate}
\item[{\rm (i)}] $E$ is a $\kappa$-Fr\'{e}chet--Urysohn space;
\item[{\rm (ii)}] $E^\lambda$ is $\kappa$-Fr\'{e}chet--Urysohn for every cardinal $\lambda>0$;
\item[{\rm (iii)}] $E$ is an Ascoli space;
\item[{\rm (iv)}] $E^\lambda$ is  Ascoli for every cardinal $\lambda>0$;
\item[{\rm (v)}] $E$ is a sequentially Ascoli space;
\item[{\rm (vi)}] $E^\lambda$ is  sequentially Ascoli  for every cardinal $\lambda>0$;
\item[{\rm (vii)}] $X$ is either finite or has the property $(\kappa)$;
\item[{\rm (viii)}] $E^\lambda$ is $\kappa$-Fr\'{e}chet--Urysohn (Ascoli or sequentially Ascoli) for some cardinal $\lambda>0$.
\end{enumerate}
\end{theorem}

\begin{proof}
If $X$ is finite, then $E=Y^{|X|}$ is metrizable. Therefore $E^\lambda$ is $\kappa$-Fr\'{e}chet--Urysohn for every cardinal $\lambda>0$, see \cite{LiL} or \cite{Gabr-B1}. Now we assume that $X$ is infinite.

The equivalences (i)$\Leftrightarrow$(iii)$\Leftrightarrow$(v)$\Leftrightarrow$(vii) follow from Theorem \ref{t:Cp-seq-Ascoli} and Proposition \ref{p:Cb-Yp-approx}.
The implication (ii)$\Rightarrow$(iv) follows from Theorem 2.5 of \cite{Gabr-B1}, 
and  (iv)$\Rightarrow$(vi) is clear.
The implications (ii)$\Rightarrow$(i), (iv)$\Rightarrow$(iii) and (vi)$\Rightarrow$(v) are also clear.

Let us prove that (vii) implies (ii). Assume that $X$ has the property $(\kappa)$. Then, by Proposition 7.3 of  \cite{Sak2}, the topological sum $Z=\bigoplus_\lambda X$ of $\lambda$ copies of the space $X$ also has the property $(\kappa)$. Therefore, by the equivalence (i)$\Leftrightarrow$(vii), the space $C^{rc}_p(Z,Y)$ is $\kappa$-Fr\'{e}chet--Urysohn. It is clear that the space $C^{rc}_p(Z,Y)$ is naturally homeomorphic to a dense subset of the locally convex space $E^\lambda$. Thus  $E^\lambda$ is $\kappa$-Fr\'{e}chet--Urysohn by Corollary 2.2 of \cite{Gabr-B1}.

The implications (ii)$\Rightarrow$(viii), (iv)$\Rightarrow$(viii) and (vi)$\Rightarrow$(viii) are clear. Finally, the implication (viii)$\Rightarrow$(i) ((viii)$\Rightarrow$(iii) or (viii)$\Rightarrow$(v)) follows from the fact that $E$ is a direct summand of $E^\lambda$ and Proposition 3.3 of \cite{LiL} (or Corollary \ref{c:open-image-Ascoli}, respectively).\qed
\end{proof}

Now we give an application of Theorem \ref{t:Cp-seq-Ascoli} to spaces of Baire functions. Let $X$ and $Y$ be topological spaces.  We say that a sequence $\{f_n\}_{n\in\w}\subseteq Y^X$ {\em stably converges} to a function $f\in Y^X$ if for every $x\in X$ the set $\{n\in\w:f_n(x)\not= f(x)\}$ is finite.  Set $B_0(X,Y)=B_0^{st}(X,Y):=C_p(X,Y)$. For every countable nonzero ordinal $\alpha$,  denote by $B_\alpha(X,Y)\subseteq Y^X$ ($B_\alpha^{st}(X,Y)\subseteq Y^X$) the family of all functions $f:X\to Y$ which are pointwise limits of function sequences from $\bigcup_{\beta<\alpha}B_\beta(X,Y)$ ($\bigcup_{\beta<\alpha}B_\beta^{st}(X,Y)$, respectively). The family $B(X,Y):=\bigcup_{\alpha\in\w_1} B_\alpha(X,Y)$ is the class of all Baire functions from $X$ to $Y$. It is clear that
\[
C_p(X,Y) \hookrightarrow B_1^{st}(X,Y) \hookrightarrow B_1(X,Y) \hookrightarrow B(X,Y).
\]

For the first countable spaces $X$ the following theorem was proved in (a) of Theorem~3.5 in \cite{Gabr-B1}.
\begin{theorem} \label{t:Baire-kFU}
Let $Y$ be a metrizable abelian group, $X$ be a $Y$-Tychonoff space, and let $H$ be a subgroup  of $Y^X$ containing $B_1^{st}(X,Y)$. Then $H$ is a $\kappa$-Fr\'{e}chet--Urysohn space.
\end{theorem}

\begin{proof}
Observe that $B_1^{st}(X,Y)$ is a dense subgroup of $H$. Hence, by Corollary 2.2 of \cite{Gabr-B1}, if  $B_1^{st}(X,Y)$ is a $\kappa$-Fr\'{e}chet--Urysohn space then so is $H$. Therefore it is sufficient to show that $H=B_1^{st}(X,Y)$ is a $\kappa$-Fr\'{e}chet--Urysohn space.

It is well known (see for example Proposition 5.3 in \cite{GO}) that the space $B(X,Y)$ and hence its subgroup $H$ can be considered as a subspace of $C_p(X_{\aleph_0},Y)$, where $X_{\aleph_0}$ is the $P$-modification of $X$ (recall that $X_{\aleph_0}$ is zero-dimensional). Further, Proposition 5.15 of \cite{GO} states that for every finite subset $D$ of $Y$, the intersection $H\cap C_p(X_{\aleph_0},D)$ is a relatively $D_p$-Tychonoff subspace of $C_p(X_{\aleph_0},D)$. In particular, the group $H$ is a relatively $Y_p$-Tychonoff subspace of $C_p(X_{\aleph_0},Y)$. Note also that the space $X_{\aleph_0}$ is $Y$-Tychonoff (because it is zero-dimensional) and has the property $(\kappa)$ by Proposition \ref{p:kappa-P-space}. Now we consider two cases.

{\em Case 1. The group $Y$ is discrete.} Then $H$ is $\kappa$-Fr\'{e}chet--Urysohn by Proposition \ref{p:Cp-kappa-FU-discrete}.

{\em Case 2. The group $Y$ is not discrete.} By Example \ref{exa:zero-dim-Yp-approx} and the aforementioned Proposition 5.15 of \cite{GO}, the group $H$ has also the $p$-approximation property. Now Proposition \ref{p:Cp-kappa-FU} implies that $H$ is a $\kappa$-Fr\'{e}chet--Urysohn space. \qed
\end{proof}

Let $X$ be a Tychonoff space and $E$ be a locally convex space. A map $f:X\to E$ is called {\em bounded} ({\em relatively compact}) if the image $f(X)$ is a bounded (respectively, relatively compact) subset of $E$. For every countable ordinal $\alpha$, we denote by $B^{b}_\alpha(X,E)$, $B^{rc}_\alpha(X,E)$, $B^{st,b}_\alpha(X,E)$ or $B^{st,rc}_\alpha(X,E)$ the family of all functions from $B_\alpha(X,E)$ or $B^{st}_\alpha(X,E)$ which are bounded or relatively compact, respectively. It is clear that
\[
B^{st,rc}_\alpha(X,E) \subseteq B^{st,b}_\alpha(X,E) \subseteq B^{st}_\alpha(X,E) \; \mbox{ and } \; B^{rc}_\alpha(X,E) \subseteq B^{b}_\alpha(X,E) \subseteq B_\alpha(X,E).
\]
The next assertion generalizes (b) of Theorem~3.5 in \cite{Gabr-B1}.

\begin{theorem} \label{t:Baire-kFU-lcs}
Let $E$ be a metrizable locally convex space, $X$ be a Tychonoff space, and let $H$ be a subgroup  of $E^X$ containing $B_1^{st,rc}(X,E)$. Then $H$ is a $\kappa$-Fr\'{e}chet--Urysohn space.
\end{theorem}

\begin{proof}
As in the proof of Theorem \ref{t:Baire-kFU}, we can assume that $H=B_1^{st,rc}(X,E)$ and that $H$ can be considered as a subgroup of $C_p(X_{\aleph_0},E)$. By the aforementioned Proposition 5.15 of \cite{GO}, the group $H$ is a relatively $E_p$-Tychonoff subspace of $C_p(X_{\aleph_0},E)$ and has the $p$-approximation property. Since $X_{\aleph_0}$ is a Tychonoff $P$-space, Proposition \ref{p:kappa-P-space} implies that $X_{\aleph_0}$ has the property $(\kappa)$.  Finally, Proposition \ref{p:Cp-kappa-FU} implies that $H$ is a $\kappa$-Fr\'{e}chet--Urysohn space. \qed
\end{proof}




Now we consider the following problem:
\begin{problem} \label{prob:Ck-seq-Ascoli}
Let $Y$ be a metrizable space (for example, $Y=\IR, \mathbf{2}$ or $[0,1]$). Characterize $Y$-Tychonoff spaces $X$ for which the function space $\CC(X,Y)$ is (sequentially) Ascoli.
\end{problem}

In two partial and important cases when $Y=\mathbf{2}$ and $X$ is a zero-dimensional metric space or $Y=\IR$ and $X$ is a metrizable space, Problem \ref{prob:Ck-seq-Ascoli} is solved completely in \cite{GKP} and \cite{Gabr-C2}, respectively (the clauses (a) in (i) and (ii) follow from Proposition \ref{p:Ascoli-sufficient} and the proofs of Lemma 3.2 in \cite{GGKZ} and Lemma 2.5 in \cite{Gabr-C2}, respectively).
\begin{theorem}
\begin{enumerate}
\item[{\rm (i)}] {\rm (\cite{GGKZ})} If $X$ is a paracompact space of point-countable type, then the following assertions are equivalent: (a) $\CC(X)$ is sequentially Ascoli, (b) $\CC(X)$ is  Ascoli, (c) $\CC(X,\II)$ is sequentially Ascoli, (d) $\CC(X,\II)$ is  Ascoli, (e) $X$ is locally compact.
\item[{\rm (ii)}] {\rm (\cite{Gabr-C2})} For a metrizable zero-dimensional space $X$ the following assertions are equivalent: (a) $\CC(X,\mathbf{2})$ is sequentially $2$-Ascoli, (b) $\CC(X,\mathbf{2})$ is  Ascoli, (c) $X$ is locally compact or $X$ is not locally compact but the set $X'$ of non-isolated points of $X$ is compact.
\end{enumerate}
\end{theorem}

For the general case of an arbitrary Tychonoff space $X$ we are able to prove the next necessary condition.
\begin{proposition} \label{p:seq-Ascoli-Ck-necessary}
If a Tychonoff space  $X$ is such that $\CC(X)$ is sequentially Ascoli, then every compact-finite sequence of compact subsets of $X$ has a strongly compact-finite subsequence.
\end{proposition}

\begin{proof}
Let $\{ K_n\}_{n\in\w}$ be a compact-finite sequence of nonempty compact sets in $X$. We have to show that there is a subsequence of $\{ K_n\}_{n\in\w}$ which is strongly compact-finite. For every $n\in\w$,  define a function $F_n$ on $\CC(X)$ by 
\[
F_n(f):=\| f-(n+1)\cdot\mathbf{1}\|_{K_n} =\sup\big\{ |f(x)-(n+1)|: x\in K_n\big\}.
\]
Since $F_n$ is the composition of the restriction map $\CC(X)\to \CC(K_n)$ and the norm of the Banach space $C(K_n)$, we obtain that $F_n$ is continuous. For every $n\in\w$,  define
\[
A_n:=\left\{f\in \CC(X): F_n(f)\leq \tfrac{1}{8}\right\} \;  \mbox{ and }\; O_n:=\left\{f\in \CC(X): F_n(f)< \tfrac{1}{4} \right\}.
\]
so $A_n$ is a functionally closed subset of $\CC(X)$ and $O_n$ is a functionally open neighborhood of $A_n$.

{\em Claim 1. $\mathbf{0}\in \overline{\bigcup_{k\in\w} A_{n_k}} \setminus \bigcup_{n\in\w} A_{n_k}$, for every sequence $n_0<n_1<\cdots$ in $\w$. So each subsequence of $\{ A_n\}_{n\in\w}$ is not locally finite.} Indeed, let $[K;\e]$ be a standard neighborhood of the zero-function $\mathbf{0}$ in $\CC(X)$, where $K\subseteq X$ is compact and $\e>0$. As $\{ K_n\}_{n\in\w}$ is compact-finite, there is an $m\in\w$ such that $K\cap K_{n_m}=\emptyset$. Then any function $h\in \CC(X)$ such that $h{\restriction}_K=0$ and $h{\restriction}_{K_{n_m}}=n_m+1$ belongs to $[K;\e]\cap A_{n_m}$. The claim is proved.
\smallskip

Now, since $\CC(X)$ is sequentially Ascoli,  Claim 1 and (vi) of Theorem  \ref{t:seq-R-Ascoli} imply that the sequence $\{ O_n\}_{n\in\w}$ is not compact-finite. Therefore there is a compact subset $\KK$ of $\CC(X)$ such that the set $J:=\{ n\in\w: \KK\cap O_n\not=\emptyset\}$ is infinite. For every $j\in J$, choose a function $h_j\in \KK\cap O_j$ and set
\[
U_j:= \{ x\in X: |h_j(x) -(j+1)|<\tfrac{1}{2}\}.
\]
Note that $U_j$ is an open  neighborhood of $K_j$. Therefore to show that the sequence  $\{ K_j\}_{j\in J}$ is strongly compact-finite it is sufficient to prove that the sequence $\U=\{ U_j\}_{j\in J}$ is compact-finite. To this end, fix a compact subset $K$ of $X$ and let $\pi_K: \CC(X)\to \CC(K)$ be the restriction map. Set $J_0:=\{ j\in J: U_j\cap K\not=\emptyset\}$ and suppose for a contradiction that $J_0$ is infinite. For every $j\in J_0$, set $g_j:=\pi_K(h_j)$. Then $g_j\in \pi_K(\KK)$, where $\pi_K(\KK)$ is a compact subset of the Banach space $C(K)$. Therefore there is $d>0$ such that $\| g_j\|_K \leq d$ for every $j\in J_0$. However, since $K\cap U_j\not=\emptyset$ by our supposition, we obtain that $\| g_j\|_K \geq j+\tfrac{1}{2} \to\infty$. This contradiction shows that $J_0$ is finite. Thus $\U$ is compact-finite.\qed
\end{proof}

We end the paper with the following  problem.
\begin{problem}
Is the converse assertion in Proposition \ref{p:seq-Ascoli-Ck-necessary} true? Is it true that the space $\CC(X)$ is Ascoli if and only if it is $\kappa$-Fr\'{e}chet--Urysohn?
\end{problem}

A characterization of $\kappa$-Fr\'{e}chet--Urysohn spaces $\CC(X)$ is given by Sakai \cite{Sakai-3}.


\bibliographystyle{amsplain}

\end{document}